\documentclass[12pt]{article}
\usepackage{amsmath}
\usepackage{amsfonts}
\usepackage{amssymb}
\usepackage{amsthm}
\usepackage{mathabx}
\usepackage{graphicx}
\usepackage{cite}
\usepackage{geometry}
\usepackage{indentfirst}
\usepackage[T1]{fontenc}
\usepackage{lmodern}

\newtheorem{theorem}{Theorem}[section]
\newtheorem{corollary}[theorem]{Corollary}
\newtheorem{lemma}[theorem]{Lemma}
\newtheorem{prop}{Proposition}[section]

\theoremstyle{remark}
\newtheorem{remark}{Remark}[section]

\def\eps{\varepsilon}
\def\R{\mathbb{R}}

\emergencystretch=\maxdimen
\title{A priori estimates and exact solvability for non-coercive stochastic control equations}
\author{Maria Luísa Pasinato, Boyan Sirakov}
\date{}

\begin{document}
\maketitle
\begin{abstract}
    We establish, for the first time, explicit a priori and regularity estimates for solutions of the Dirichlet problem for Hamilton-Jacobi-Bellman operators from stochastic control, whose  principal half-eigenvalues have opposite signs. In addition, if the negative eigenvalue is not too negative, the  problem can have exactly two, one or zero solutions, depending on the valuation function. This is a novel exact multiplicity result for fully nonlinear equations, which also yields a generalization of the Ambrosetti-Prodi theorem to such equations.
    \end{abstract}

\section{Introduction}
\par This paper is a contribution to the study of the solution set of the Dirichlet problem
\begin{equation}
\label{Main_Dirichlet}
        F(D^2u,Du,u,x) = g \text{ in } \Omega, \qquad
        u = 0 \text{ on } \partial\Omega,
\end{equation}
where $F: \mathcal{S}^N \times \R^N \times \R \times \Omega \mapsto \R $ is a Hamilton-Jacobi-Bellman (HJB) operator with bounded measurable coefficients, $\Omega \subset \R^N$ is a bounded $C^{1,1}$-domain, and $g \in L^p(\Omega)$ for  $p>N$.

An important motivation for studying  equations of HJB type is their connection to problems of optimal control and the associated stochastic differential equations (SDE). The classical Bellman dynamic programming method stipulates that  the optimal cost (value) function  of a controlled SDE should be a solution of  \eqref{Main_Dirichlet}. For a thorough exposition of this subject we refer to the books \cite{flemingmete}, \cite{krylovbook}; here we recall that if the stochastic process $X_t$ is subjected to the law
$$
dX_t = b^{\alpha_t}(X_t) dt + \sigma^{\alpha_t}(X_t)dW_t\,,
$$
with $X_0 = x$, for some $x\in\Omega$, and we are given the cost/value function
$$
J(x,\alpha) = \mathbb{E} \int_0^{\tau_x} g^{\alpha_s}(X_t) \exp \left\{\int_0^t c^{\alpha_s}(X_s) ds\right\}\,dt,
$$
where $\tau_x$ is the first exit time from ${\Omega}$ of $X_t$, $W_t$ is the standard Brownian motion,  $b^{\alpha_t}$ are drift vectors, $\sigma^{\alpha_t}$ are positive definite diffusion matrices, $c^{\alpha_t}$ are stochastic discount rates, $g^{\alpha_t}$ are the valuation functions; all these parameters depend on  an index (control) process $\alpha_t$  with values in a set ${\cal A}$, then the optimal cost function
$v(x) = \inf_{\alpha\in {\cal A}} J(x,\alpha)$ is such that $u=-v$ is  a solution of \eqref{Main_Dirichlet} in the following particular form (setting $A^\alpha=(\sigma^\alpha)^T\sigma^\alpha$)
\begin{equation}\label{equ1}
\left\{\begin{array}{rclcc}
\displaystyle\sup_{\alpha\in {\cal A}} \{\mathrm{tr} (A^\alpha(x)D^2u) + b^\alpha(x).Du
+ c^\alpha(x) u - g^\alpha(x)\} &=& 0 &\mbox{ in }& \Omega\\ u&=&0 &\mbox{on}& \partial
\Omega.\end{array}\right.
\end{equation}
We will mostly require  that the coefficients in this PDE be bounded measurable; we stress however that all results below are new even for equations with smooth coefficients.

Most works on HJB equations deal with {\it proper} equations,
 meaning that $F$ in \eqref{Main_Dirichlet} is non-increasing in the variable $u$. In the well-known works
\cite{lions1981acta}, \cite{evans1982classical}, \cite{evans1983classical},   it was proved that a proper HJB equation of type \eqref{equ1}
has a unique smooth solution, if the coefficients are
smooth. Results on existence and uniqueness of viscosity solutions for proper HJB equations with bounded measurable ingredients can be found in \cite{ishii1989uniqueness},
\cite{crandall1992user},
\cite{caffarelli1996viscosity},   \cite{swiech1997w}, \cite{crandall1999existence}, \cite{crandall2000lp}, \cite{dongkrylov}.

\par In order to go beyond properness, at least when $F$ in \eqref{Main_Dirichlet} is positively homogeneous in the unknown function $u$ (such as the operator in \eqref{equ1} with $g^\alpha=0$), it is nowadays well understood that one may consider the so-called principal half-eigenvalues $\lambda_1^+$ and $\lambda_1^-$ of $F$ in $\Omega$
\begin{equation}\label{eigvs}
        \lambda_1^{\pm}(F,\Omega) = \sup\{ \lambda \in \mathbb{R} : \exists\;  \psi \in W^{2,p}_{\mathrm{loc}}(\Omega), \quad \pm(F[\psi]+\lambda \psi) \leq 0, \quad\pm \psi > 0 \text{ in } \Omega\,\}.
\end{equation}
Of course, if $F=L$ is linear ($|\cal A|= $1 in \eqref{equ1}), then $\lambda_1^{+}(F,\Omega) = \lambda_1^{-}(F,\Omega) = \lambda_1(L,\Omega)$ -- see \cite{berestycki1994principal}.

 The positivity of $\lambda_1^+$ and $\lambda_1^-$ (which is much more general than properness) implies existence and uniqueness of solutions to \eqref{Main_Dirichlet}. For HJB operators with smooth
coefficients this was proved in  \cite{lions1983bifurcation}, via  probabilistic and analytic
techniques. More recently an approach inspired by  the seminal work  \cite{berestycki1994principal} (on the principal eigenvalue of a linear operator), was used in \cite{birindelli2006eigenvalue},  \cite{birindelli2007eigenvalue}, \cite{QUAAS2008105} and \cite{armstrong2009principal}, where many more properties of the half-eigenvalues $\lambda_1^\pm$ and related existence and uniqueness results were established -- some of these will be recalled and used below. Also, if instead of positive, both half-eigenvalues are negative but not too distant from $0$, then the Dirichlet problem \eqref{Main_Dirichlet} is again uniquely  solvable for every right-hand side $g$ -- see \cite{armstrong2009principal}, \cite{siam2010}. These results extend to positively homogeneous HJB equations the classically known existence and uniqueness fact for linear equations given by the Fredholm alternative.

 On the other hand, the possibility of having principal half-eigenvalues with opposite signs represents a purely nonlinear configuration, one that has no linear analogue. The questions whether $\lambda_1^+ < 0 < \lambda_1^-$ also implies existence or uniqueness of solutions to \eqref{Main_Dirichlet} were both answered in the negative in \cite{sirakovnon} (see also \cite{demengel}, \cite{siam2010} for extensions). The result in \cite{sirakovnon} shows that, contrary to what happens in the linear case,  if a HJB equation is not at resonance and $0$ is in the ``nonlinear spectral gap'' $(\lambda_1^+, \lambda_1^-)$, the solution set of the Dirichlet problem is bounded but may be empty or have more than one element.

All previous works on HJB equations in the case $\lambda_1^+ < 0 < \lambda_1^-$ treat existence, non-existence and multiplicity of solutions. However, another fundamental feature in elliptic PDE theory -- explicit a priori/regularity estimates in terms of norms of the coefficients and the right-hand side $g$ -- has been absent in the literature. Our first main goal here is to provide such estimates, which also quantify the proximity of $\lambda_1^+, \lambda_1^-$ to zero. Specifically, we obtain an ABP-type $L^\infty$-estimate for the solutions of the Dirichlet problem. It is remarkable that explicit ABP-type bounds are usually associated with situations where uniqueness of solutions holds, and to our knowledge have been proven only in such cases, whereas here we prove such an estimate in the presence of multiple solutions.

Further, we also obtain estimates on the subset of the space of right-hand sides, for which solutions exist. Finally, we give essentially optimal hypotheses under which an exact multiplicity result of Ambrosetti-Prodi type can be proved for HJB equations, thus obtaining a very detailed description of the solution set. These will be explained in more detail next.

In  \cite{sirakovnon} it was suggested that there is a connection between the  nonuniqueness when $\lambda_1^+ < 0 < \lambda_1^-$  and the apparently different Ambrosetti-Prodi (AP) phenomenon,  named after the celebrated result in \cite{ambrosetti1972inversion} on a semilinear equation involving the Laplacian. It states that, given a strictly convex and Lipschitz $f\in C^2(\R)$ such that ${f'(\R)}$ intersects the spectrum of the Laplacian only at its first eigenvalue, there exists a map $t^* : C^{0,\alpha}(\Omega) \rightarrow \R$ for which
\begin{equation}\label{ambrpreq}
    \Delta u + f(u) = g= h + t\varphi_1 \text{ in } \Omega, \qquad
    u = 0 \text{ on } \partial\Omega
\end{equation}
has no solutions when $t<t^*(h)$, exactly one solution when $t=t^*(h)$, and exactly two when $t>t^*(h)$. In other words, we can split  $ C^{2,\alpha}(\Omega)\cap C_0(\overline{\Omega})=  A \cup \mathcal{C}\cup B$, $C^\alpha(\Omega)=S_0 \cup F(\mathcal{C}) \cup S_2$ where $F(\cdot)=\Delta \cdot + f(\cdot)$, in such a way that $\mathcal{C}$ is homeomorphic to a hyperplane, and both $A$ and $B$ are taken by $F$  homeomorphically to $S_2$. This {\it exact multiplicity} result (such results are not very frequent in the PDE theory)  has generated a lot of work and extensions over the years. We refer to \cite{bergerpodolak}, \cite{solimini}, \cite{rufsurvey}, \cite{calanchitomei}, as well as the many references in these works, for even more precise descriptions of the action of a general semilinear self-adjoint operator $F$.

It is worth observing that the AP result was extended to  linear operators in  non-divergence form (that is, $\Delta$ in \eqref{ambrpreq} replaced by an operator like in \eqref{equ1} with $|\cal A|= $1) only recently in \cite{sirakov2018results}. In that paper the linearity of the second-order operator was  crucial in order to divide the domain of $F$ into {\it fibers}, the analysis being also based on methods from the regularity theory of elliptic PDE in non-divergence form. We  refer to \cite{sirakov2018results} for a  historical review and further discussion of Ambrosetti-Prodi type results for linear $F$.

Here we will  show that the AP phenomenon still occurs for nonlinear HJB operators, i.e. arbitrary suprema of linear elliptic operators. This has been completely open up to now and is somewhat unexpected.  Theorem \ref{Main_Result} below is the first result whatsoever in the literature which establishes exact count of solutions (different from 0 or 1) for equations governed by fully nonlinear operators. The exact mupltiplicity also makes it possible to describe the form of the solution set. Differently from previous works on  exact multiplicity, we use topological degree theory.

To be more precise, we recall the main result in  \cite{sirakovnon}:
for a HJB operator with  $\lambda_1^+(F,\Omega) < 0 < \lambda_1^-(F,\Omega)$ and first half-eigenfunction $\phi=\varphi_1^+>0$, there exists $t^*: L^{\infty}(\Omega) \rightarrow \R$ such that
\begin{equation}
\label{Main_Decomp}
        F(D^2u,Du,u,x) = g(x)= h(x) + t\phi(x) \text{ in } \Omega, \qquad
        u = 0 \text{ on } \partial\Omega
        \tag{$\mathcal{P}_t$}
\end{equation}
has no solutions when $t<t^*(h)$, {\it at least one solution} when $t=t^*(h)$, and {\it at least two} when $t>t^*(h)$. This is what degree theory usually gives, (non-)existence of solutions, but without specifying their exact number  (we refer to the  survey \cite{bandle2004solutions} which contains many references).
We develop a variant of the degree method to yield exact multiplicity of solutions -- zero,  one or  two. We also give a general framework where $\lambda_1^+ < 0 < \lambda_1^-$ is required only from an asymptotic approximation of $F$, thus making the  Ambrosetti-Prodi theorem  a very particular instance of our results.

As the above references show, the existence of a critical number $t^*$ is nowadays very standard in AP type problems. On the other hand,
a feature completely absent from previous works, even in the simplest cases and despite its obvious importance for applications, is an explicit bound for $t^*(h)$ in terms of $h$ and the constants governing the behaviour of $F$.
Actually, as we noted above,
explicit a priori bounds (in terms of $\lambda_1^+$,$\lambda_1^-$ and $\|h\|_p$) do not seem to be available not just for $t^*(h)$ but even for  the solutions of \eqref{Main_Decomp} themselves.
We provide such estimates here,  only assuming $\lambda_1^+ < 0 < \lambda_1^-$ , without additional hypotheses on the expected number of solutions of the Dirichlet problem.\medskip

To illustrate the main theorems below, we  state them now in the frequently encountered particular  case of a control problem \eqref{equ1} in which one uses the same valuation function in all states (that is, $g^\alpha=g$ is the same for all $\alpha$). The general case of \eqref{equ1} is covered by the results in the next section, where more detailed results, examples, and discussion are given.
\begin{theorem}\label{APHJB}
    Let $g\in L^p(\Omega)$ for some $p>n$. Assume there exist $0<\lambda\leq \Lambda$, for which $A_{\alpha} \in C(\mathcal{S}^N)$ satisfy $\lambda I \leq A_{\alpha} \leq \Lambda I$, and $||b_{\alpha}||_{L^{\infty}(\Omega)}, ||c_{\alpha}||_{L^{\infty}(\Omega)},  \leq \Lambda$, for all $\alpha\in{\cal A}$. Let
    \begin{equation} \label{eq:HJB_def}
        F(D^2u,Du,u,x) = \sup_{\alpha \in \mathcal{A}} \left( \text{Tr}(A^{\alpha}(x)D^2u)+ b^{\alpha}(x).Du + c^{\alpha}(x)u\right),\;\mbox{ and assume }
    \end{equation}
   \begin{equation}\label{maincond}
        \lambda_1^+:= \lambda_1^+(F,\Omega) < 0 < \lambda_1^-(F,\Omega)=:\lambda_1^-.
    \end{equation}
  I.  There is $C_0=C_0(n,\lambda,\Lambda, \Omega)>0$ such that  the following bounds are valid for ($\mathcal{P}_t$)
     \begin{equation}\label{estimateu} \|u\|_{W^{2,p}(\Omega)}\le C_0\left(1-\frac{1}{\lambda_1^+}\right)
    \left(1+\frac{1}{\lambda_1^-}\right)\|g\|_{L^p(\Omega)}\,,
    \end{equation}
      \begin{equation}\label{estimatetstar}
    -C_0\left(1-\frac{1}{\lambda_1^+}\right)
    \left(1+\frac{1}{\lambda_1^-}\right)\|h\|_{L^p(\Omega)}\le t^*(h)\le C_0\|h^-\|_{L^p(\Omega)}.
      \end{equation}
 II. There is $ d = d(n,\lambda,\Lambda,\Omega)>0$  such that if    $
        -d < \lambda_1^+(F,\Omega) < 0 < \lambda_1^-(F,\Omega)$ then
   \begin{itemize} \item \eqref{Main_Decomp} has no solutions for $t<t^*(h)$;  \item \eqref{Main_Decomp} has either exactly one solution or a segment of ordered solutions  for  $t=t^*(h)$ ; \item  \eqref{Main_Decomp} has exactly two ordered solutions
  for  $t>t^*(h)$.\end{itemize}
\end{theorem}

\noindent \textbf{Acknowledgement.} The authors are indebted to Carlos Tomei for many insightful conversations and suggestions.

\section{Main results and Examples}
\par In this section, we state our assumptions, enunciate our main results in their  general form, and comment on some immediate applications.

Given some $F,H : \mathcal{S}^N \times \R^N \times \R \times \Omega \mapsto \R$ with $F(0,0,0,x)=0$, we may require that they satisfy a {\it subset} of the following hypotheses.
\begin{enumerate}
    \item[H0($F$):] $F$ is \textit{positively homogeneous} of degree one in its first three entries: if $t\geq0$ and $ (M,p,u,x) \in \Upsilon := \mathcal{S}^N \times \R^N \times \R \times \Omega$ then
    \begin{equation*}
        F(tM,tp,tu,x) = tF(M,p,u,x).
    \end{equation*}
    \item[H1($F$):] $F$ is \textit{uniformly elliptic and satisfies a Lipschitz structure condition}:  there exist $0<\lambda < \Lambda$ and $\gamma, \delta >0$ such that  for any $(M,p,u,x),(N,q,v,x) \in \Upsilon$
  \begin{equation*}
       \mathcal{L}^-(M-N,p-q,u-v)\leq F(M,p,u,x) - F(N,q,v,x)
        \leq  \mathcal{L}^+(M-N,p-q,u-v),
    \end{equation*}
  where $\mathcal{L}^\pm(M,p,u) = \mathcal{M}^\pm_{\lambda,\Lambda}(M) \pm\gamma|p| \pm\delta|u|$, and $\mathcal{M}^\pm$ are the Pucci extremal operators.

    \item[H2($F$):]  $F(\cdot,0,0,\cdot) : \mathcal{S}^N \times \Omega \mapsto \R$ is \textit{continuous}.

    \item[C($F$):] $F$ is \textit{convex} in its first three entries:   for  $(M,p,u,x), (N,q,v,x) \in \Upsilon $ and $ \alpha \in [0,1]$
        \begin{equation*}
        F(\alpha M + (1-\alpha)N, \alpha p + (1-\alpha) q, \alpha u + (1-\alpha) v, x) \leq \alpha F(M,p,u,x) + (1-\alpha) F(N,q,v,x).
    \end{equation*}

    \item[sC($F$):] $F$ is \textit{strictly convex} in its first three entries around $u=0$: there is $\eta>0$ such that
    \begin{equation*}
        F(\alpha M + (1-\alpha)N, \alpha p + (1-\alpha) q, \alpha u + (1-\alpha) v, x) < \alpha F(M,p,u,x) + (1-\alpha) F(N,q,v,x)
    \end{equation*}
    for any $(M,p,x), (N,q,x) \in \mathcal{S}^N \times \R^N \times \Omega $
    and  $u,v\in (-\eta,0)$ or $u,v\in (0,\eta)$; $ \alpha \in (0,1)$.
    \item[D$_H(F)$:] For all $(M,p,u,x), (N,q,v,x)\in \Upsilon $
    \begin{equation*}
        F(M,p,u,x) - F(N,q,v,x) \leq H(M-N,p-q,u-v,x).
    \end{equation*}
    \item[U$_H(F)$:] There exists $A_0 \in \R$ such that  for all $ (M,p,u,x) \in \Upsilon$
    \begin{equation*}
F(M,p,u,x) \geq H(M,p,u,x) - A_0.
\end{equation*}

\end{enumerate}

When H0($F$) is valid we say $F$ satisfies H0, and similarly for the other hypotheses. When no confusion arises, we often write $F[u] = F(D^2u,Du,u,\cdot)$, for simplicity. In what follows, all differential (in)equalities are assumed to hold in the $L^p$-viscosity sense (see \cite{crandall2000lp}). We recall that if $u \in W^{2,p}_{\mathrm{loc}}(\Omega)$, satisfying an inequality in the $L^p$--viscosity sense is equivalent to satisfying it for a.e. $x \in \Omega$. We write $\|\cdot\|_p=\|\cdot\|_{L^p(\Omega)}$.

Hypothesis H0 is positive 1--homogeneity, hypothesis H1 is a joint Lipschitz continuity and uniform ellipticity assumption, and hypothesis H2 imposes continuity in the spatial variable at least on the leading terms of the operator. Together with hypothesis C (convexity), H1  characterizes the behavior of Hamilton-Jacobi-Bellman operators as in \eqref{equ1},  whereas C-H0-H1 characterize \eqref{eq:HJB_def}.  Hypothesis sC is a slightly stronger convexity assumption, which we will use only to guarantee uniqueness of solutions when $t=t^*(h)$. Hypothesis D$_H(F)$ is somewhat standard and appears naturally in many contexts; for instance, H1($F$) is equivalent to $\text{D}_{\mathcal{L}^+}(F)$ and, when H0($F$) holds, C($F$) is equivalent to $\text{D}_F(F)$ (see \cite{armstrong2009principal}).

We will generally work with an operator $F$ satisfying H1 and C.
For any such operator we can  define its asymptotic approximation (see Lemma \ref{deffinftylem} below)
\begin{equation}\label{conv1}
    F_\infty(M,p,u,x) = \sup_{t \in \R^+} \frac{1}{t} F(tM,tp,tu,x) = \lim_{t \to \infty} \frac{1}{t} F(tM,tp,tu,x)
\end{equation}
Then $F_\infty$ also satisfies H1 and C and, in addition, $F_\infty$ satisfies H0 and D$_{F_{\infty}}(F)$ holds. We will also require U$_{F_{\infty}}(F)$ which gives some uniformity in the convergence in \eqref{conv1}, as well as H2$(F_\infty)$. Thus, our standing assumption will be
\begin{enumerate}
    \item[SC($F$):] \ \
$F: \mathcal{S}^N \times \R^N \times \R \times \Omega \mapsto \R$  is such that   $H1(F),  C(F), H2(F_\infty), U_{F_{\infty}}(F)$  hold.
\end{enumerate}
For the reader's convenience we record that (SC) means  $F$ is uniformly elliptic, Lipschitz and convex,  its asymptotic approximation $F_\infty$ is such too, $F_\infty$ is  positively $1$-homogeneous in $(M,p,u)$, $F_\infty$ verifies (H2), D$_{F_{\infty}}(F)$ holds, and $F_\infty-A_0\le F\le F_\infty$ (all results here are new in the particular case $A_0=0$, $F=F_\infty$ too).
Thus, if $F$ satisfies (SC), Theorem~\ref{ABP} below applies to its asymptotic approximation $F_\infty$ and gives principal half-eigenvalues of $F_\infty$.

Note that the optimal control problem \eqref{equ1} is \eqref{Main_Dirichlet} with
$$
    F(M,p,u,x) = \sup_{\alpha \in \mathcal{A}} \left( \mathrm{tr} (A^\alpha(x)M) + b^\alpha(x).p
+ c^\alpha(x) u - f^\alpha(x) \right) + \inf_{\alpha \in \mathcal{A}} f^{\alpha}(x),
$$
$ g(x) = \inf_{\alpha \in \mathcal{A}} f^{\alpha}(x)$, and then $ F_{\infty}(M,p,u,x) = \sup_{\alpha \in \mathcal{A}} \left( \text{Tr}(A^{\alpha}(x)M)+ b^{\alpha}(x).p + c^{\alpha}(x)u \right)
$ is exactly the operator in  \eqref{eq:HJB_def}. So for this operator $F$ the assumption  (SC) means $A^{\alpha}$ are uniformly positive and continuous, whereas all other coefficients are bounded measurable, these properties being uniform in $\alpha$. Also, $A_0 = 2\sup_{\alpha \in \mathcal{A}} \|f^{\alpha}\|_{L^\infty(\Omega)}$.

We  state the general a priori bound for operators with principal half-eigenvalues of opposite sign. From now on $c_0,C_0$  denote constants depending only on $n,p,\lambda,\Lambda,\gamma,\delta$, $\Omega$.
\begin{theorem}\label{aprbd} Let $g\in L^p(\Omega)$ for some $p>n$. Assume $F$ satisfies (SC), and   \begin{equation}
    \label{opposite_sign} \lambda_1^+:= \lambda_1^+(F_{\infty},\Omega) < 0 < \lambda_1^-(F_{\infty},\Omega)=:\lambda_1^-.\end{equation}
    Then for any viscosity solution $u$ of
    \begin{equation}\label{genereq}
   -g^-(x)\le  F(D^2u, Du, u,x) \le g^+(x)\; \mbox{in }\Omega, \qquad u=0 \; \mbox{on }\partial\Omega
    \end{equation}
    we have
    \begin{equation}\label{generest}
    \|u\|_{L^\infty(\Omega)}\le C_0\left( 1-\frac{1}{\lambda_1^+}\right)\left( 1+\frac{1}{\lambda_1^-}\right) (A_0+\|g\|_{L^p(\Omega)}).
    \end{equation}
\end{theorem}
Combining  \eqref{generest} with standard $W^{2,p}$ estimates which hold for solutions (see Theorem \ref{W2P}) yields $W^{2,p}$ bounds like  \eqref{estimateu}.
\begin{remark} The second inequality in \eqref{genereq} alone implies the expected ABP-bound for $u^-$ (Theorem \ref{ABP} 3. below), while the first inequality alone implies no bound at all ($u=k\varphi_1^\pm$ solves $F[u]\ge0$ for any $k\in\mathbb{R}$).
We will actually prove a  more precise a priori estimate, showing that the terms containing $\lambda_1^\pm$ in \eqref{generest} multiply only $\|g^+\|$, and even that there is $\sigma_0>0$ depending only on $n,p,\lambda, \Lambda, \gamma, \delta$, and $\Omega$  such that for every $\rho\in(0,1]$
\begin{equation}\label{moregenerest}
\|u\|_\infty\le \rho\|g^-\|_p + C_0 \rho^{-\sigma_0}\left( 1-\frac{1}{\lambda_1^+}\right)\left( 1+\frac{1}{\lambda_1^-}\right) (A_0+\|g^+\|_p).
\end{equation}
This also  quantifies the fact that the only solution of $F_\infty[u]\le 0$ is $u=0$ (see Lemma \ref{heq0}).
\end{remark}

To describe better the solution set of \eqref{Main_Dirichlet} we decompose the right-hand side
\begin{equation}
\label{Main_Decomp1}
        F(D^2u,Du,u,x) = h(x) + t\phi(x) \text{ in } \Omega, \qquad
        u = 0 \text{ on } \partial\Omega
        \tag{$\mathcal{P}_t$}
\end{equation}
Here, $h \in L^p(\Omega)$ for some $p>N$ and $t \in \R$, where $\phi=\varphi_1^+\in W^{2,p}(\Omega) $, $\phi>0$,  is the first eigenfunction of $F_\infty$ given by Theorem \ref{ABP} below, normalized so that $\|\phi\|_{L^\infty(\Omega)}=1$. We will denote the  solution set of \eqref{Main_Decomp} by $\mathcal{S}_t $ and set $\mathcal{S}_I = \bigcup_{t \in I} \mathcal{S}_t$ for $I \subset \R$. Whenever we need to deal with \eqref{Main_Decomp} for different $h$, we specify by writing $(\mathcal{P}_{t,h})$, $\mathcal{S}_{t,h}$ and $\mathcal{S}_{I,h}$. In practice it is convenient to restrict $h$ to a hyperplane in $L^p(\Omega)$ (for instance $\int_\Omega h=0$) or to a manifold homeomorphic to that plane (for instance $\|h^-\|_p = \|h^+\|_p$ which balances nicely the $L^p$-norms in the inequalities) and then study the ``minimal height" $t^*(h)$ for which \eqref{Main_Decomp1} has solutions.

 The following theorem estimates the extremal value $t^*(h)$.

\begin{theorem}\label{aprbdtstar} Let $h\in L^p(\Omega)$ for some $p>n$. Assume $F$ satisfies (SC) and  $$   \lambda_1^+:= \lambda_1^+(F_{\infty},\Omega) < 0 < \lambda_1^-(F_{\infty},\Omega)=:\lambda_1^-.$$
    The problem \eqref{Main_Decomp1} has a solution if and only if
   $t\in[t^*(h), \infty)$, with
    \begin{equation}\label{generesttstar}
   -C_0\left(1-\frac{1}{\lambda_1^+}\right)
    \left(1+\frac{1}{\lambda_1^-}\right)(A_0+\|h\|_{L^p(\Omega)})\le t^*(h)\le C_0\|h^-\|_{L^p(\Omega)}.
    \end{equation}
\end{theorem}

Finally, we state the extension of the Ambrosetti-Prodi result to fully nonlinear  equations.
\begin{theorem}
\label{Main_Result}
    Let $h\in L^p(\Omega)$ for some $p>n$. Assume $F$ satisfies (SC). Then there is $\beta = \beta(n,p,\lambda,\Lambda,\gamma,\delta,\Omega) >0$ for which \begin{equation}
    \label{eq:lambdas_change_sign} -\beta < \lambda_1^+(F_{\infty},\Omega) < 0 < \lambda_1^-(F_{\infty},\Omega),\end{equation} implies that
    \begin{enumerate}
        \item when $t < t^*(h)$, \eqref{Main_Decomp} has no solutions;
        \item when $t > t^*(h)$, \eqref{Main_Decomp} has exactly two strictly ordered solutions $\underline{u}_t < \overline{u}_t$;
        \item when $t=t^*(h)$, the solution set $\mathcal{S}_t$ is the line segment connecting  $\underline{u} = \lim_{t\searrow t^*} \underline{u}_t$ and $\overline{u} = \lim_{t\searrow t^*} \overline{u}_t$. If  sC(F) holds, \eqref{Main_Decomp} has exactly one solution $\underline{u} =\overline{u}$.
    \end{enumerate}
   Moreover, the map $h \mapsto t^*(h)$ is continuous from $L^p(\Omega)$ to $\R$, and the maps $t \mapsto \underline{u}_t$ and $t \mapsto \overline{u}_t$ are continuous from $(t^*,+\infty)$ to $C^1(\overline{\Omega})$. The map  $t \mapsto \underline{u}_t$ is strictly pointwise decreasing, while $t \mapsto \overline{u}_t$ is strictly increasing in the (measure) sense of Proposition \ref{upper_curve_nondecreasing} below; also, $\max_K \underline{u}_t\to-\infty$ and $\min_K \overline{u}_t\to+\infty$ as $t\to\infty$, for each compact $K\subset\Omega$, and $\underline{u}_t<0$, $\overline{u}_t>0$ in $\Omega$ for large $t$.
\end{theorem}

 Our arguments suggest interpreting half-eigenvalues as merely a sort of homogeneous skeleton of a nonlinearity. As an example, consider the semilinear operator
\begin{equation}\label{ap1}
        F(M,u) = \text{Tr}(M) + f(u),
\end{equation}
for some Lipschitz and convex $f$ (this is the classical Ambrosetti-Prodi case). Then
\begin{equation}\label{ap2}
    F_{\infty}(M,u) = \text{Tr}(M) - au^- + bu^+
\end{equation}
where $a = \inf f'(\R)$ and $b = \sup f'(\R)$. If we assume that $a<\lambda_1 = \lambda_1(-\Delta, \Omega)<b$, then
\begin{equation}\label{ap3}
   \lambda_1^+(F_{\infty}) =  \lambda_1 -
b <0, \qquad \lambda_1^-(F_{\infty}) =\lambda_1 - a>0,
\end{equation}
so Theorem~\ref{Main_Result} applied to \eqref{ap1} contains the classical Ambrosetti-Prodi result. More generally, if $F(M,p,u,x) = \text{Tr}(A(x)M) + b(x) \cdot p + c(x)u + f(u)$,
Theorem \ref{Main_Result}  yields the main result from \cite{sirakov2018results}.

The  Fu\v{c}ik operator \eqref{ap2} can be used to explain why the appearance of the product $|\lambda_1^+\lambda_1^-|^{-1}$ in the estimates \eqref{generest}-\eqref{generesttstar} is natural. If $u$ is a solution of the Dirichlet problem for
$$
\Delta u+bu^+-au^- = h,\quad \mbox{ or }\quad \Delta u+bu^+= \widetilde{h}_u:= h+au^- \quad \mbox{in }\Omega,
$$
the first of these implies $\|u^-\|_\infty\le C_0(\lambda_1-a)^{-1}\|h\|_p$ (see Theorem \ref{ABP} 3. below) while the second implies $\|u^+\|_\infty\le C_0(b-\lambda_1)^{-1}\|\widetilde{h}_u\|_p$ for $b$ close to $\lambda_1$ (this is a variant of the {\it anti-maximum principle}, see \cite{hess}, \cite{birindelliantimax}, \cite{armstrong2009principal}). Of course, for general $F$ there is no such simple argument, the proofs of Theorems \ref{aprbd}-\ref{aprbdtstar} rely on a quantitative estimate on a possibly negative first eigenvalue, which is interesting in its own right  (and appears to be new even in the linear case)-- see Proposition~\ref{estnegeig} below.

We next discuss the hypotheses of Theorem \ref{Main_Result} and their optimality. Actually, their necessity can already be observed in the (simplest) semilinear case \eqref{ap1}-\eqref{ap3}.

First,  an exact multiplicity result like Theorem \ref{Main_Result} cannot be obtained without a restriction as in \eqref{eq:lambdas_change_sign}: if $\beta>\lambda_2(-\Delta, \Omega)-\lambda_1(-\Delta, \Omega)$, that is,  $f^\prime(\mathbb{R})$ contains more than one eigenvalue of the Laplacian, the Dirichlet problem for $\Delta u + f(u) = g$ can have more than two solutions -- see \cite{solimini} and the references there. For  $F$ in non-divergence form the optimal value of $\beta$ is not known, even in the linear case (but a lower bound on $\beta$ in terms of the constants in the standard elliptic estimates can be obtained).

Further, without an additional hypothesis as in Theorem \ref{Main_Result} 3. it is indeed possible to have infinity (a segment) of solutions at the critical value $t=t^*(h)$. This is the case for instance if in \eqref{ap1}-\eqref{ap3} we have  $f^\prime = \lambda_1$ in a left or right neighborhood of $0$. Indeed, if $f(u) = \lambda_1u +l$ for $u\in (0,L)$, then $t\varphi_1$ solves $\Delta u + f(u) = -l$ for $t\in (0,L)$.

Finally, we observe that convexity is in general unavoidable for a result like Theorem \ref{Main_Result} to hold. In \cite{calanchitomeizaccur} it was shown that even if $f^{\prime\prime}$ is negative just at one point, there are right hand sides for which the Dirichlet problem for \eqref{ap1}  has at least four solutions.

We  note that it should be only a matter of technicalities to weaken the  integrability assumption on $g$ (resp. $h$), to $g\in L^p(\Omega)$ for $p>p_E$ where $p_E=p_E(n,\lambda, \Lambda)\in (n/2,n)$ is the constant from \cite{Esc}, and $g\in L^p$ for $p>n$ only in a neighborhood of $\partial \Omega$ (to guarantee Lipschitz and Hopf type estimates at the boundary). Also, the continuity assumption on $F_\infty(M,0,0,x)$ can be weakened to a VMO-type hypothesis (as in  \cite{swiech1997w}, \cite{crandall2000lp}, \cite{winter2009w}, \cite{dongkrylov}), which we have not written, for simplicity.

The next section is preliminary and contains a few already known results which we need. The proofs of our results can be found in Section \ref{secproofs}.

\section{Preliminaries}

 We start with
 the following global $W^{2,p}$ estimate, which follows from \cite[Theorem 1.2]{dongkrylov}(see also Remarks 1.4-1.6 in that paper, as well as  \cite{swiech1997w}, \cite{crandall2000lp}, \cite{winter2009w} for earlier results).

\begin{theorem}[\cite{dongkrylov}]
\label{W2P}
Assume (SC). There exists a constant $C = C(n,p,\lambda,\Lambda, \gamma, \delta, A_0, \Omega)$ such that, if $u$ is a (viscosity) solution of \eqref{Main_Dirichlet}
for some $g \in L^p(\Omega)$, then $u \in W^{2,p}(\Omega)$ and
\begin{equation}\label{regest}||u||_{W^{2,p}(\Omega)} \leq C(||u||_{L^\infty(\Omega)} + ||g||_{L^p(\Omega)}).
\end{equation}
\end{theorem}
\begin{remark}
    The results in \cite{dongkrylov} assume $F$ to be nonincreasing in $u$, i.e. proper, and in that case {\it the unique solvability of \eqref{Main_Dirichlet} is also established}. Then \eqref{regest} follows trivially, by observing that H1(F) implies properness for $F-\delta$ and by considering the equation $F[u]-\delta u = g - \delta u$.
\end{remark}

So all viscosity solutions of equations (but not necessarily of inequations) in this paper will be {\it strong} solutions, more precisely belonging to $W^{2,p}(\Omega)$ and satisfying the equation almost everywhere in $\Omega$. Since $p>N$, this also implies solutions are in $ C^{1,\alpha}(\overline{\Omega})$ for $\alpha=1-n/p$, by Sobolev embeddings.\medskip

We turn to the existence and properties of principal half-eigenvalues, and the
related one-sided  generalized ABP estimates obtained in  \cite{QUAAS2008105} and \cite{armstrong2009resonance}.
\begin{theorem}[\cite{QUAAS2008105}]
\label{ABP}
 Let $\Omega$ be an arbitrary bounded domain in $\mathbb{R}^n$.   Assume $F$ satisfies H0, H1, H2 and C.
    \begin{enumerate}
    \item The numbers $ \lambda_1^+\le\lambda_1^-$ in \eqref{eigvs} are simple and isolated eigenvalues of $F$ in $\Omega$,  associated to a positive
and a negative eigenfunctions
$\varphi_1^\pm\in W^{2,q}_{\text{loc}}(\Omega)\cap C_0(\overline{\Omega}) $ (even $W^{2,q}(\Omega)$ if $\Omega$ is smooth), for   $q<\infty$. These numbers are monotone and continuous with respect to~$\Omega$. Also, $-\delta\le \lambda_1^+\le\lambda_1^-\le C_0/r^2$, if $\Omega$ contains a ball with radius $r\in(0,1)$.
\item Assume there exists a viscosity solution
$u\in C(\overline{\Omega})$ of one of the problems
\begin{equation}\label{simp1}
   F(D^2u, Du, u, x) =  -\lambda_1^+ u\quad\mbox{in }\;\Omega,\qquad
u=0\quad\mbox{on }\;\partial\Omega;
\end{equation}
\begin{equation}\label{simp2}
 F(D^2u, Du, u, x) \le  -\lambda_1^+ u\quad\mbox{in }\;\Omega, \qquad u>0\quad\mbox{in }\;\Omega;
\end{equation}
\begin{equation}\label{simp3}
 F(D^2u, Du, u, x) \ge  -\lambda_1^+ u\quad\mbox{in }\;\Omega, \qquad u(x_0)>0, \qquad u\le0\quad\mbox{on }\;\partial\Omega;\end{equation}
for some $x_0\in \Omega$. Then $u\equiv t\varphi_1^+$, for some $t\in
\mathbb{R}$. If a function $u\in C(\overline{\Omega})$ satisfies
either (\ref{simp1}) or the inverse inequalities in (\ref{simp2}) or (\ref{simp3}),
with $\lambda_1^+$ replaced by $\lambda_1^-$, then $u\equiv t\varphi_1^-$.
        \item If $\lambda_1^-(F,\Omega) > 0$, then for some $C = C(n,\lambda,\Lambda,\gamma,\delta,\Omega)$ and any $u \in C^0(\overline{\Omega}) $, $g \in L^N(\Omega)$,
        $
            F(D^2u,Du,u,x) \leq g
        $
        implies
        \begin{equation}\label{abpminus}
            \sup_{\Omega} u^- \leq C\left(1+\frac{1}{\lambda_1^-}\right)\left(\sup_{\partial\Omega} u^- + ||g^+||_{L^N(\Omega)}\right).
        \end{equation}
        \item If $\lambda_1^+(F,\Omega) > 0$, then for some $C = C(n,\lambda,\Lambda,\gamma,\delta,\Omega)$ and any $u \in C^0(\overline{\Omega}) $, $g \in L^N(\Omega)$,
       $
            F(D^2u,Du,u,x) \geq g
       $
        implies
        \begin{equation}\label{abpplus}
            \sup_{\Omega} u \leq C\left(1+\frac{1}{\lambda_1^+}\right)\left(\sup_{\partial\Omega} u^+ + ||g^-||_{L^N(\Omega)}\right).
        \end{equation}
    \end{enumerate}
\end{theorem}
  \begin{remark}
\label{Unif_ABP}
The explicit form of the constants in \eqref{abpminus} and \eqref{abpplus} is due to \cite[Theorem 2.3]{armstrong2009resonance}.\end{remark}
\begin{corollary}[\cite{QUAAS2008105}]
\label{Comparison_Corollary}
      Assume $F$ satisfies H0, H1, H2 and C, and $\text{D}_F(H)$ holds. Then, if $\lambda_1^+(F,\Omega) >0$, $H$ satisfies the comparison principle in the following sense: whenever $u, v \in C^0(\Omega)$ satisfy
    \begin{align}
    \label{Comparison}
            H(D^2u,Du,u,x) \leq  H(D^2v,Dv,v,x) \text{ in } \Omega, \qquad
            u \geq v \text{ on } \partial \Omega,
    \end{align}
    in the $L^p$--viscosity sense, and at least one of them is in $W^{2,p}_{\text{loc}}(\Omega)$, then $u \geq v$ in $\Omega$.
\end{corollary}


Next, we quote an existence result based on positivity of half-eigenvalues.
\begin{theorem}[\cite{QUAAS2008105}]
\label{Positive_Existence}
    Assume $F$ satisfies H0, H1, H2 and C, and  $\lambda_1^+(F,\Omega) > 0$. In this case, given any $f \in L^p(\Omega)$, there exists a unique solution $u \in W^{2,p}(\Omega)$ of \eqref{Main_Dirichlet}.
    \par Alternatively, suppose $\lambda_1^-(F)>0$. Then, for any $g \in L^p(\Omega)$ such that $g\geq 0$ a.e. in $\Omega$, there exists a nonpositive solution $u \in W^{2,p}(\Omega)$ of \eqref{Main_Dirichlet}.
\end{theorem}
\begin{remark}
    The statements in \cite{QUAAS2008105} give solutions merely belonging to $W^{2,p}_{\text{loc}}(\Omega)\cap C^0(\overline{\Omega})$. That solutions are actually in  $W^{2,p}(\Omega)$ follows for instance from Theorem \ref{W2P}.
\end{remark}

\begin{remark}\label{properex} A particular case of operators for which the first half-eigenvalues are positive are proper operators, and a fortiori, operators independent of $u$. Given $F(M,p,u,x)$ as above, we denote $F_0(M,p,x)=F(M,p,0,x)$.
We recall (see \cite{caffarelli1996viscosity}) that the problem
\begin{equation}\label{puccieq}
\mathcal{L}^\pm_0(D^2u, Du) =\mathcal{M}^\pm_{\lambda,\Lambda}(D^2u) \pm\gamma|Du| = g \in L^n(\Omega)
\end{equation}
has a unique solution such that $u=0$ on $\partial\Omega$ and any sub(resp. super)-solution $u\in C(\overline{\Omega})$ of \eqref{puccieq} satisfies for some $C_0=C_0(n,\lambda,\Lambda, \gamma, \mathrm{diam}(\Omega))$
\begin{equation}\label{pucciest}
\sup_{\Omega} u \leq \sup_{\partial\Omega} u^+ + C_0 ||g^-||_{L^N(\Omega)}\qquad \left(\mbox{resp. }\; \sup_{\Omega} u^- \leq \sup_{\partial\Omega} u^- + C_0 ||g^+||_{L^N(\Omega)}\right).
\end{equation}
The same is valid for $F_0$ instead of $\mathcal{L}^\pm_0$, if $F$ satisfies (H1). Similarly, under H1 the operator  $F_\delta(M,p,u,x) = F(M,p,u,x)-\delta u $ is proper and the same existence statements and estimates are valid for (sub-, super-)solutions of $F_\delta[u]=h$.  \end{remark}

 We will use the following quantitative strong maximum principle, or  global weak Harnack inequality, from \cite{sirakovimrn}.
\begin{theorem}[\cite{sirakovimrn}]
\label{SMP}
    Suppose $g\in L^p(\Omega)$ for some $p>n$, and $u$ is a (viscosity) solution of
    \begin{equation}\label{smpeq}
        \mathcal{M}_{\lambda,\Lambda}^-(D^2u) - \gamma |Du| - \delta |u| \leq g(x), \qquad \mbox{and} \qquad u\geq 0\;\mbox{ in }\;\Omega.
    \end{equation}
    Then for some $p_0, c_0>0$ depending on $n,\lambda,\Lambda, \gamma, \delta$ and $\Omega$ (recall $d=d(x)=\mathrm{dist}(x,\partial\Omega)$)
    \begin{equation}\label{smpres}
        \inf_\Omega \frac{u}{d}\ge c_0 \Big\|\frac{u}{d}\Big\|_{p_0} - c_0^{-1} \|g^+\|_p,\quad\mbox{and}\quad \inf_\Omega \frac{u}{d}\ge c_0\|g^-\|_{p_0} - c_0^{-1} \|g^+\|_p
    \end{equation}
\end{theorem}
\begin{remark}\label{remhopf} Theorem \ref{SMP} quantifies the usual Hopf-Oleinik strong maximum principle (proved first for viscosity solutions in \cite{bardi1999strong}). We recall the latter says that if $g=0$ in \eqref{smpeq} then  $u \equiv 0$ or $u>0$ in $\Omega$ and  we have $\lim \inf_{t\searrow 0} t^{-1}(u(x+t\nu) - u(x)) > 0$, $x \in \partial \Omega$, where $\nu$ is the unit interior normal to $\partial \Omega$ at $x$. In particular, if $F$ satisfies H1 and $F[u_1]\le F[u_2]$, $u_1\ge u_2$ in $\Omega$, and one of $u_1$, $u_2$ is in  $W^{2,p}_{\text{loc}}(\Omega)$, then $u=u_1-u_2$ satisfies the conclusion of the Hopf-Oleinik strong maximum principle. \end{remark}

\begin{remark}  Theorem \ref{SMP} is stated in \cite{sirakovimrn} for $\delta = 0$ but the same proof trivially extends to the case $\delta>0$. \end{remark}
\begin{remark}\label{unifdist} The eigenfunction $\phi$ in \eqref{Main_Decomp} satisfies $(F_\infty+\lambda_1^+)[\phi]=0$ so by the bound on $\lambda_1^+$ in Theorem \ref{ABP} 1., the normalization $\|\phi\|_\infty=1$ and the $C^1$ bound we have $\|\phi\|_{C^1(\Omega)}\le C_0$. This implies, on one hand, that $\phi\le C_0d$ in $\Omega$, and on the other hand, that $\phi\ge 1/2$ in a ball in $\Omega$ with radius $1/(2C_0)$ which together with \eqref{smpres} (with $g=0$) implies that $\phi\ge c_0d$ in $\Omega$ for some $c_0>0$. The same applies to any bounded solution $w$ of a Dirichlet problem with $C^1$ estimates and operator satisfying (H1), yielding $0<c_0d\le w\le C_0d$ in $\Omega$. \end{remark}

We next quote a quantitative estimate on the monotonicity of $\lambda_1^+$ with respect to $\Omega$, obtained  in \cite[section 9]{berestycki1994principal} (see the proof of Theorem 2.4 there; the adjustments needed for HJB operators are listed in \cite[Proposition 6.1]{siam2010}), which  plays a pivotal role in the proof of Theorem \ref{Main_Result}.
\begin{theorem}[\cite{berestycki1994principal},\cite{siam2010}]
\label{Monotone_Eigenvalue}
    Assume $F$ satisfies H0, H1, H2 and C. Let $\Gamma$ be a closed subset of $\Omega$ satisfying $|\Gamma| > a$. Then, there exists a constant $\beta_1 = \beta_1(N,\lambda,\Lambda,\gamma,\delta,\Omega,a)>0$ such that we have $\lambda_1^+(F,\Omega\setminus\Gamma) > \lambda_1^+(F,\Omega) + \beta_1$.
\end{theorem}
\begin{remark}
    It should be stressed that the dependence of $\beta_1$ on the subset $\Gamma$ in the above theorem occurs only through a lower bound on the measure of $\Gamma$.
\end{remark}
Finally, we state well-known general results on consistency of viscosity (sub and super)solutions (see \cite{caffarelli1996viscosity}), which will also be frequently used in our arguments.
\begin{theorem}
\label{Consistency}
    Assume $F$ satisfies H1. If $f_n \rightarrow f$ in $L^p(\Omega)$ and $u_n \rightarrow u$ in $C^0(\overline{\Omega})$ satisfy $F(D^2u_n,Du_n,u_n,x) \leq(\geq) f_n$, then u satisfies $F(D^2u,Du,u,x) \leq(\geq) f$.
\end{theorem}
\par An application of Theorems \ref{Consistency} and \ref{W2P} gives us a simple continuity result which will often suffice in place of the general consistency statement. We state this as a corollary.
\begin{corollary}
\label{Consistency_Cor}
    Assume (SC). If $\{u_n\}$ is a bounded (in $L^{\infty}$) sequence satisfying $F(D^2u_n,Du_n,u_n,x) = f_n$, with $f_n \rightarrow f$ in $L^p(\Omega)$, then there exists a subsequence of $\{u_n\}$ which converges in $C^1(\overline{\Omega})$ to a solution of $F(D^2u,Du,u,x) = f$.
\end{corollary}
\begin{proof}
    If $\{u_n\}$ is bounded in $L^{\infty}(\Omega)$, then Theorem \ref{W2P} implies it is also bounded in $W^{2,p}(\Omega)$. By the Sobolev  embedding of $W^{2,p}(\Omega)$ into $C^{1,1-n/p}(\overline{\Omega})$, there exists a subsequence of $\{u_n\}$ converging in $C^1(\overline{\Omega})$ to some $u_{\infty}$. Theorem \ref{Consistency} then yields $F(D^2u_{\infty},Du_{\infty},u_{\infty},x) = f$.
\end{proof}
\par We will need (pre-)compactness for sequences of functions which are related to solutions of \eqref{Main_Dirichlet}, but are not solutions themselves. This is a consequence from the classical  Krylov-Safonov  H\"older bound (in our setting it follows for example from Proposition 4.2 in \cite{crandall1999existence} or the proof of  Theorem 2 in \cite{sirakovarma}).
\begin{theorem}
\label{Compactness}
If $f\in L^n(\Omega)$ and $u \in C^0(\overline{\Omega})$ is a viscosity solution to
\begin{equation}\label{calph}
\mathcal{L}^-(D^2u,Du,u) \leq f, \qquad \mathcal{L}^+(D^2u,Du,u) \geq -f, \qquad u=0 \text{ on } \partial\Omega,
\end{equation}
then $u\in C^\alpha(\Omega)$ for some $\alpha>0$ depending  on $n,\lambda,\Lambda,\gamma,\delta,\Omega$, and $\|u\|_{C^\alpha(\Omega)}\le C_0\|f\|_{L^n(\Omega)}$. Hence the set of all solutions of \eqref{calph}, for $f$ in a fixed ball in $L^n(\Omega)$, is precompact in $C^0(\overline{\Omega})$.
\end{theorem}

Next, moving closer to our setting here, we give a simple lemma which in particular verifies  that the asymptotic approximation $F_{\infty}$ satisfies the necessary hypotheses for our statements to make sense.
\begin{lemma}\label{deffinftylem} Assume $F: \mathcal{S}^N \times \R^N \times \R \times \Omega \mapsto \R$  satisfies H1 and C. Then $F_\infty(M,p,u,x) = \sup_{t>0} t^{-1}F(tM,tp,tu,x)$ is well-defined,
\begin{equation}\label{deffinfty}
F_\infty(M,p,u,x) = \lim_{t\rightarrow \infty} \frac{1}{t}F(tM,tp,tu,x),
\end{equation}
 and $F_\infty$ satisfies (H0), (H1), (C), $\text{D}_{F_{\infty}}(F)$.
\end{lemma}

\begin{proof} That \eqref{deffinfty}, (H0),  (C), $\text{D}_{F_{\infty}}(F)$ hold follows directly from Lemma \ref{app1} in the Appendix. The condition H1($F_{\infty}$) follows from taking the limit of the respective inequalities in (H1) for~$F$ (note that in H1($F$) the left and right hand sides are positively 1-homogeneous).
\end{proof}
\begin{remark} We observe that if $F$ satisfies (H2) and U$_{F_{\infty}}(F)$  holds, then $F_\infty$ also satisfies (H2).
This follows from a simple $\eps / 3$ argument: consider
    \begin{equation}
    \begin{split}
        |F_{\infty}(M,0,0,x) - F_{\infty}(N,0,0,y) | \leq &|F_{\infty}(M,0,0,x) - t^{-1}F(tM,0,0,x) |\\  & + | t^{-1}F(tM,0,0,x) - t^{-1}F(tN,0,0,y)| \\ &+ |t^{-1}F(tN,0,0,y) - F_{\infty}(N,0,0,y)|.
    \end{split}
    \end{equation}
    Taking $t=\varepsilon A_0^{-1}/3$ and applying hypothesis U($F$) makes the first and last terms smaller than $\varepsilon /3$. Hypothesis H2($F$) then allows us to choose $(M,x),(N,y)$ so close that the middle term is also bounded by $\varepsilon / 3$.
    However, we will only assume U$_{F_{\infty}}(F)$ and H2(F$_\infty$) which is a weaker assumption than H2(F). For instance, under our hypotheses the operator in \eqref{equ1} does not satisfy H2 but U$_{F_{\infty}}(F)$ and H2(F$_\infty$) hold.
    \end{remark}

We record  the following simple and useful convexity property.
    \begin{lemma}
    \label{convex_subsuper}
        Suppose C($F$) and let $u,v \in W^{2,p}(\Omega)$ be solutions of \eqref{Main_Dirichlet}. Then, $u + \alpha(v-u)$ is a supersolution of the same problem for $\alpha \in [0,1]$ and a subsolution for $\alpha \in (-\infty,0]\cup[1,\infty)$. If in addition for some $\alpha_0\in(0,1)$ the function $u + \alpha_0(v-u)$ is a solution of \eqref{Main_Dirichlet}, then $u + \alpha(v-u)$ is a solution of the same problem for all $\alpha \in [0,1]$.
    \end{lemma}
    \begin{proof} This follows from Lemma \ref{app2} in the Appendix. \end{proof}

We finish this preliminary section with a result from \cite{sirakovnon} on existence of super-solutions of our problem. We give a quick proof for the reader's convenience.

\begin{lemma}
    \label{supersol_exist}
 Assume $F$ satisfies (SC).   There exists $T_0 = T_0(n,\lambda, \Lambda, \gamma, \delta,\Omega) $ such that, for every $t > T_0\|h^-\|_p$ there exists a supersolution $\overline{v} \in W^{2,p}(\Omega)$ of \eqref{Main_Decomp} satisfying $\overline{v} \geq 0$ in $\Omega$.
\end{lemma}
\begin{proof}
    Let $\overline{v}$ be the unique solution of the Dirichlet problem
    \begin{equation}
            \mathcal{M}_{\lambda,\Lambda}^+(D^2\overline{v}) + \gamma |D\overline{v}| = -h^- \text{ in } \Omega, \qquad
            \overline{v} = 0 \text{ on } \partial\Omega,
    \end{equation}
which exists by Corollary 3.10 in \cite{caffarelli1996viscosity}, see Remark \ref{properex} above. Theorem \ref{W2P} and \eqref{pucciest} imply $\overline{v} \geq 0$ and $||\overline{v}||_{C^1} < C_0\|h^-\|_p$, hence $v(x)\le C_0\|h^-\|_p \, d(x)$, where $d(x) = \mathrm{dist}(x,\partial\Omega$).
Since $\phi\ge c_0\,d$ by Remark \ref{unifdist},  H1($F$) implies
\begin{equation*}
    F[\overline{v}] \leq \mathcal{M}_{\lambda,\Lambda}^+(D^2\overline{v}) + \gamma |D\overline{v}| + \delta \overline{v}=- h^-+ \delta \overline{v}\leq h + t\phi,
\end{equation*}
as long as $t>\delta c_0^{-1}C_0\|h\|_p$, so $\overline{v}$ is a supersolution of  \eqref{Main_Decomp}.
\end{proof}


\section{Proofs of the main results}\label{secproofs}

The a priori bounds stated in Theorems \ref{aprbd} and \ref{aprbdtstar} will be proved with the help of the following proposition, which gives a lower bound on a possibly negative first eigenvalue of a HJB operator, in terms of a quantified failure of existence of a positive supersolution. This result should clearly be useful in other contexts, and has, to our knowledge, not appeared even in the case when $F$ is a linear operator.

\begin{prop}\label{estnegeig} Assume $F$ satisfies H0, H1, H2 and C (so Theorem \ref{ABP} applies to $F$). Given $\alpha_0, N_0>0$, there exist positive constants $\epsilon_0, C_0$ depending on $n,p,\lambda, \Lambda, \gamma, \delta, \alpha_0, N_0$, and~$\Omega$, such that, if there exist functions $w\in C^{\alpha_0}(\Omega)$ and $\epsilon\in L^p_+(\Omega)$ for which
$$
\|w\|_{C^{\alpha_0}(\Omega)}\le N_0, \qquad\qquad F[w]\le \epsilon(x)\text{ in } \Omega, \qquad
            \|w\|_{\infty} \ge m_0, \qquad \|w^-\|_{\infty} \le \|\epsilon\|_p\le \epsilon_0m_0,
$$
for some $m_0\in(0,1]$, then
$$
 \lambda_1^+(F,\Omega)\ge - (C_0/m_0) \|\epsilon\|_p.
$$
\end{prop}
The H\"older bound is a mild hypothesis, it holds for any solution of an equation or even solutions of two-sided  inequalities as in Theorem \ref{Compactness}; and then $N_0$ depends on the usual quantities (actually, any uniform modulus of continuity for $w$ can be used instead of the H\"older one). The important assumption is that we can find a supersolution $w$ for a small right-hand side, sich that $w$ stays away from zero but its negative part is small relative to the positive part, with a quantification of the smallness; and this leads to a quantified lower bound on the possibly negative first eigenvalue of the operator.

\begin{proof} By replacing $w$ by $w/m_0$ and $\epsilon$ by $\epsilon/m_0$ we can assume $m_0=1$. Also, by replacing $w$ by $w+\|\epsilon\|_p$ and $\epsilon$ by $\epsilon + \delta \|\epsilon\|_p$ and by using H1, we can assume that $w\ge0$ in $\Omega$. If $\epsilon =0$ then by the definition of $\lambda_1^+$ we have $\lambda_1^+(F,\Omega)\ge0$; so we can assume that $\epsilon \gneqq0$.

If $x_0$ is a point such that $w(x_0)\ge1$ then by the H\"older bound $w\ge 1/2$ in the ball $B_{r_0}(x_0)$, for $r_0= (2N_0)^{-1/\alpha_0}$. Hence by Theorem \ref{SMP}
$$
\inf_\Omega\frac{w}{d} \ge c_0 r_0^{n/p_0} - c_0^{-1}\|\epsilon\|_p = c_1 - c_0^{-1}\epsilon_0\ge c_1/2>0,
$$
if we choose $\epsilon_0 =c_0c_1/2$.

Let $\psi$ be the solution of the problem
$$
\mathcal{L}_0^+[\psi] = -\frac{\epsilon(x)}{\|\epsilon\|_p}\;\text{ in }\; \Omega, \qquad \psi = 0 \;\text{ on } \; \partial \Omega.
$$
As we recalled in the previous section (Remark \ref{properex}), this problem has a unique solution $\psi>0$. By \eqref{pucciest} and the Lipschitz ($C^1$) bound $\psi\le C_0d$ in $\Omega$, so, setting $C_1=2C_0/c_1$,
$$
\psi\le C_1 w\;\text{ in }\; \Omega.
$$
Finally, by H1 again
\begin{align*}
F[w+\|\epsilon\|_p\psi] &\le F[w]+\|\epsilon\|_p\mathcal{L}^+
[\psi]= F[w]+\|\epsilon\|_p\mathcal{L}_0^+[\psi] +\delta \|\epsilon\|_p\psi \\
&\le \delta \|\epsilon\|_p\psi \le C_1\delta \|\epsilon\|_p w\le C_1\delta \|\epsilon\|_p (w+\|\epsilon\|_p\psi),
\end{align*}
so the definition of $\lambda_1^+$ in \eqref{eigvs} implies the statement.
\end{proof}
\begin{remark} \label{algrem}
We see from this proof that the dependence of $\epsilon_0, C_0$ in $N_0$ is algebraic, that is,  $C_0 = \tilde C_0 N_0^{\sigma_0}$, $\epsilon_0 = \tilde \epsilon_0 N_0^{-\sigma_0}$, for suitable $\tilde\epsilon_0, \tilde C_0, \sigma_0$ independent of $N_0$. \end{remark}

\medskip

We next give the proof of the a priori bounds on the solutions to \eqref{Main_Decomp}. \smallskip

\noindent {\it Proof of Theorem \ref{aprbd}}. First, since $F_\infty[\cdot]-A_0\le F[\cdot]\le F_\infty[\cdot]$, it is enough to assume $F=F_\infty$ and $A_0=0$.

 It immediately follows from Theorem \ref{ABP} 3. that for some $C_0\ge1$
 \begin{equation}\label{assthis}\|u^-\|_{L^\infty(\Omega)}\le C_0\left( 1+\frac{1}{\lambda_1^-}\right) \|g^+\|_{p}.
 \end{equation}

 Assume that for some $\epsilon_0\in (0,1)$
 \begin{equation}\label{assthat}
 \|u^+\|_{L^\infty(\Omega)}\ge \epsilon_0^{-1}C_0\left( 1+\frac{1}{\lambda_1^-}\right) \|g\|_{p},
 \end{equation}
 and set
 $$
 w=\frac{u}{\|u\|_\infty} = \frac{u}{\|u^+\|_\infty}, \qquad \hat\epsilon = \frac{g}{\|u\|_\infty}
 $$
 Since
 $$
 -\hat\epsilon^- \le F[w]\le \hat\epsilon^+, \qquad \mbox{and}\qquad \|\hat\epsilon\|_p\le 1
$$
the standard H\"older bounds (see Theorem \ref{Compactness}) imply that for some $\alpha_0, N_0>0$ which depend only on $n,p,\lambda,\Lambda, \gamma, \delta$ and $\Omega$, we have $\|w\|_{C^{\alpha_0}}(\Omega)\le N_0$. We now take $\epsilon_0$ in \eqref{assthat} to be the number from the previous lemma, for these  $\alpha_0, N_0$, and $m_0=1$.

Setting
$$
\epsilon = C_0\left( 1+\frac{1}{\lambda_1^-}\right) \hat\epsilon^+
$$
we immediately check that $w$ and $\epsilon$ satisfy the hypotheses of Proposition \ref{estnegeig}, so
$$
\lambda_1^+\ge - C_0\|\epsilon\|_p = -C_0\left( 1+\frac{1}{\lambda_1^-}\right) \frac{\|g^+\|_p}{\|u\|_\infty}\,,
$$
that is,
\begin{equation}\label{assth}
\|u\|_\infty\le  \frac{C_0}{|\lambda_1^+|}\left( 1+\frac{1}{\lambda_1^-}\right) \|g^+\|_p.
\end{equation}
To summarize, \eqref{assthis} holds, and either \eqref{assth} or the
 negation of \eqref{assthat} holds too. These give exactly the a priori bound on $u$ stated in  Theorem \ref{aprbd}.

The more general bound \eqref{moregenerest} is proved similarly, instead of \eqref{assthat} assuming that
\begin{equation}\label{assthat1}
 \|u^+\|_{L^\infty(\Omega)}\ge \rho\|g^-\|_{p}+\epsilon_0^{-1}C_0\left( 1+\frac{1}{\lambda_1^-}\right) \|g^+\|_{p},
 \end{equation}
 and by using Remark \ref{algrem} (now $\|\hat\epsilon\|_p\le 1/\rho$, so $\|w\|_{C^{\alpha_0}}(\Omega)\le N_0/\rho$).
 \hfill $\Box$

 Next, we record  a lemma on existence of subsolutions, essentially proved in \cite{sirakovnon}. We give a clarified and simpler proof here.
\begin{lemma}
\label{subsol_exist}
   Assume (SC).    For any $t \in \R$, there exists a subsolution $\underline{v} \in W^{2,p}(\Omega)$ of \eqref{Main_Decomp}, with $\underline{v}<0$ in $\Omega$. Moreover, given a compact interval $I \subset \R$, $\underline{v}$ can be chosen to be a subsolution of \eqref{Main_Decomp} for all $t \in I$ and to satisfy $\underline{v} < u$ for all solutions $u$ of \eqref{Main_Decomp} for $t \in I$.
\end{lemma}
\begin{proof}
    From Theorem \ref{aprbd} which we just proved, there exists $C_1>0$ such that $||u|| < C_1$ for every $u \in \mathcal{S}_I$.
 Set $M = A_0 + \sup_{t\in I}||t\phi||_{\infty} + \delta C_1+1$. By Theorems \ref{W2P} and \ref{Positive_Existence}, there is  a non-positive solution $\underline{v} \in W^{2,p}(\Omega)$ of
    \begin{equation}
            F_{\infty}[\underline{v}] = M + h^+ \text{ in } \Omega, \qquad
            \underline{v} = 0 \text{ on } \partial \Omega.
    \end{equation}
     Since $M>0$ the strong maximum principle yields $v<0$. Since $F\ge F_\infty-A_0$, the definition of $M$  immediately implies that $\underline{v}$ is a subsolution of  \eqref{Main_Decomp} for any $t \in I$. Moreover, if $u$ solves \eqref{Main_Decomp} for some $t \in I$, by $\underline{v}\le0$, $\text{D}_{F_{\infty}}(F_{\infty})$,  and the choice of $M$ we get
    \begin{align*}
        F_{\infty}[\underline{v} - u] - \delta(\underline{v}-u) &\geq F_{\infty}[\underline{v}] - F_{\infty}[u] -\delta(\underline{v}-u) \\
        & \geq (M + h^+) - (F[u]+A_0)  + \delta u\\
        & \geq M + h^+ - t\phi - h -A_0 -\delta C_1 > 0
    \end{align*}
in $\Omega$, with $\underline{v} - u = 0$ on $\partial \Omega$. The operator $F_{\infty}-\delta$ is proper,  so by Remark \ref{properex} we have $\underline{v} - u \leq 0$ in $\Omega$, and the strong maximum principle (Remark \ref{remhopf})  gives $\underline{v} < u$.
\end{proof}

Lemma \ref{supersol_exist} and Lemma \ref{subsol_exist} imply that for $t\ge T_0$ the problem \eqref{Main_Decomp} has ordered sub and super-solutions. So Perron's method, or even a simple iteration, imply that for $t\ge T_0$ a solution exists. For the reader's convenience we recall that  it is enough to solve the hierarchy of Dirichlet problems
$$
F[u_{m+1}]-\delta u_{m+1} = t\phi+h - \delta u_m \text{ in } \Omega, \qquad
            u_{m+1} = 0 \text{ on } \partial \Omega.
$$
(which can be done by the results in \cite{dongkrylov} or \cite{crandall1999existence}, since $F-\delta$ is proper), observe that by the comparison principle $\underline{v}=u_0\le u_m\le u_{m+1} \le \overline{v}$, use Theorem \ref{W2P} to extract a convergent subsequence and pass to the limit as in Corollary \ref{Consistency_Cor}.

\begin{remark}\label{remsup} We  observe that if there exists a supersolution for $(\mathcal{P}_t)$, such a function is also a (strict) supersolution for $\mathcal{P}_s, \ s>t$.  \end{remark}

Next, we give the proof of the existence and a priori bound on the critical height $t^*(h)$.

\noindent {\it Proof of Theorem \ref{aprbdtstar}}. Without restricting the generality we can take $F=F_\infty$ and $A_0=0$ again.
Assume that for some $t<0$ there is a solution $u=u_t$ of \eqref{Main_Decomp}
$$
F[u]=t\phi+ h \le h^+
$$
As above, Theorem \ref{ABP} 3. implies
 $$\|u^-\|_{L^\infty(\Omega)}\le C_0\left( 1+\frac{1}{\lambda_1^-}\right) \|h^+\|_{p}
 $$

 Assume that for some $B>0$ to be chosen we have
\begin{equation}\label{tB}
 |t|\ge B \left( 1+\frac{1}{\lambda_1^-}\right) \|h\|_p
\end{equation}
 and set
 $$v = \frac{u}{|t|}.
 $$
 Then
 $$F[v] = -\phi + \frac{h}{|t|} = :\beta(x) \le \frac{h^+}{|t|} =:\hat\epsilon (x).
 $$
 Assuming $B>C_0$, we have  $\|\beta\|_p\le C_0 + 1/B\le 2C_0$, so by Theorem \ref{W2P}  $\|v\|_{C^1}\le N_0$.

Let $\psi\in W^{2,p}(\Omega)$, $\psi>0$, be the solution (given in Remark \ref{properex}) of
$$
F_0[\psi] = -\phi\;\text{ in }\; \Omega, \qquad \psi = 0 \;\text{ on } \; \partial \Omega.
$$
By Theorem \ref{SMP} and Remark \ref{unifdist} we have $\psi\ge c_0d$ in $\Omega$, in particular $$\sup_\Omega\psi \ge c_0>0.$$ Also,
$$
\mathcal{L}^+_0[\psi-v]\ge F_0[\psi]-F_0[v]\ge F_0[\psi]-F[v] -\delta |v|\ge -\delta|v| - \frac{h^+}{|t|}
$$
so by \eqref{pucciest} and \eqref{tB}
$$
\psi(x)-v(x) \le C_0\|v\|_\infty  +\frac{C_0}{B},
$$
for each $x\in \Omega$. Hence
$$
\|v\|_\infty\ge (C_0+1)^{-1}(\psi(x)-C_0/B)= c_1(c_0-C_0/B) \ge c_0c_1/2 =:m_0>0
$$
if we choose $x$ such that $\psi(x)=c_0$ and $B\ge 2C_0/c_0$.

On the other hand, setting $\epsilon = C_0\left( 1+\frac{1}{\lambda_1^-}\right) \hat\epsilon^+$ we have
$$
\|v^-\|_\infty\le C_0\left( 1+\frac{1}{\lambda_1^-}\right)\frac{\|h^+\|_p}{|t|} = \|\epsilon\|_p.
$$
Choosing $B\ge C_0 (\epsilon_0m_0)^{-1}$, where $\epsilon_0$ is the number from  Proposition \ref{estnegeig} with $\alpha_0=1$ and $N_0$ given above, we see all assumptions of that proposition are satisfied, so
$$
\lambda_1^+\ge - C_0\|\epsilon\|_p = -C_0\left( 1+\frac{1}{\lambda_1^-}\right) \frac{\|h^+\|_p}{|t|},
$$
which writes
\begin{equation}\label{tB1}
|t|\le  \frac{C_0}{|\lambda_1^+|}\left( 1+\frac{1}{\lambda_1^-}\right) \|h^+\|_p.
\end{equation}
Thus either \eqref{tB1} or the negation of \eqref{tB} holds, or in other words,
problem \eqref{Main_Decomp} does {\it not} have solutions for
$$
t<-t_0\|h\|_p:= -C_0\left( 1-\frac{1}{\lambda_1^+}\right) \left( 1+\frac{1}{\lambda_1^-}\right) \|h\|_p.
$$
Since we already know by Lemma \ref{supersol_exist} and Lemma \ref{subsol_exist} that solutions exist for $
t>T_0\|h^-\|_p$, and that existence for $(\mathcal{P}_t)$ implies existence for $\mathcal{P}_s, \ s>t$ (Remark \ref{remsup}), the proof of Theorem~\ref{aprbdtstar}
is complete. \hfill $\Box$

For later use we record a simple consequence for the case $h=0$.
\begin{lemma}\label{heq0} Assume $F$ satisfies H0, H1, H2 and C (so Theorem \ref{ABP} applies to $F$). If $\lambda_1^+(F,\Omega) < 0 < \lambda_1^-(F,\Omega)$,  then $t^*(0)=0$.
\end{lemma}
This follows immediately from Theorem \ref{aprbdtstar} (since now $F=F_\infty$ and $A_0=0$) but can also be seen by observing that $(\mathcal{P}_t)$ has (the zero) solution for $t=0$ and if there were a solution for $t<0$ it would be strictly positive by Theorem \ref{ABP} 3. and the Hopf lemma, but that would contradict Theorem \ref{ABP} 2. since $\lambda_1^+<0$. \medskip

The rest of this section is devoted to the proof of Theorem \ref{Main_Result}. The first lemma defines the number $\beta>0$, and shows that the solution set is {\it ordered}, which will be fundamental for the remainder of the text.

\begin{lemma}
\label{ordered_sols}
     There exists some $\beta = \beta(n,\lambda,\Lambda,\gamma,\delta,\Omega)>0$ such that, if $-\beta< \lambda_1^+(F_{\infty})$, then the solution set $\mathcal{S}_t$ is strictly ordered: if $u, v \in \mathcal{S}_t$, then either $u < v$ or $u > v$ or $u \equiv v$.
\end{lemma}
\begin{proof}
   Let $u, v$ be distinct solutions of \eqref{Main_Decomp}. From $\text{D}_{F_{\infty}}(F)$, we have
    \begin{equation}
    \label{Sign_Change_Dirichlet}
            F_{\infty}[\pm(u-v)] \geq \pm(F[u]-F[v]) = 0 \text{ in } \Omega \qquad
            \pm(u-v) = 0 \text{ on } \partial\Omega.
    \end{equation}
    \par Suppose, by contradiction, that $u-v$ changes sign inside $\Omega$. If this is the case, one of the sets $\Omega^+=\{x \in \Omega : u-v > 0\}$, $\Omega^-=\{x \in \Omega : v-u > 0\}$ has positive measure smaller than $|\Omega |/2$. Without loss of generality, assume  the former. Take $\beta < \beta_1$, where $\beta_1$ is the number  from Theorem \ref{Monotone_Eigenvalue} applied with $\Gamma = \Omega \setminus \Omega^+$. We now have $\lambda_1^+(F_{\infty},\Omega^+) > \lambda_1^+(F_{\infty},\Omega) + \beta_1 > 0$. Hence we can apply Theorem \ref{ABP} 4. to the inequality in \eqref{Sign_Change_Dirichlet} in the domain $\Omega^+$,  getting
    \begin{equation}
        \sup_{\Omega^+} (u-v) \leq C\sup_{\partial \Omega^+} (u-v) = 0,
    \end{equation}
    which contradicts the definition of $\Omega^+$. This shows that either $u \leq v$ or $u \geq v$. Applying Remark \ref{remhopf} and H1($F_{\infty}$) to equation \eqref{Sign_Change_Dirichlet} in either case yields the desired strict inequalities and concludes the proof.
\end{proof}

{\it
    In what follows, we always assume $$-\beta< \lambda_1^+(F_{\infty},\Omega)<0<\lambda_1^-(F_{\infty},\Omega).$$
}

    We consolidate some of the previous results in the following corollary.
\begin{corollary}
    \label{Ordered_Bdd}
    For every $t \in \R$, the set $\mathcal{S}_t$ is (strictly) ordered and bounded in $W^{2,p}(\Omega)$. In addition, the functions \begin{equation}x \mapsto \inf_{u\in \mathcal{S}_t} u(x), \ \ x \mapsto \sup_{u\in \mathcal{S}_t} u(x) \end{equation}
    are elements of $\mathcal{S}_t$.
\end{corollary}
\begin{proof}
    The first statement follows directly from Lemma \ref{ordered_sols} and Theorems \ref{aprbd}-\ref{W2P}. The second statement follows from a standard Perron-type argument, which we recall here. We deal only with $\overline{u} :  x \mapsto \sup_{u\in \mathcal{S}_t} u(x)$, since the other case is symmetrical. First, notice that, by definition of the supremum, there exists, for some fixed point $x_0 \in \Omega$, a sequence $u_{n,x_0}\in \mathcal{S}_t$ such that $u_{n,x_0}(x_0) \nearrow \overline{u}(x_0)$. By Lemma \ref{ordered_sols}, the sequence is monotonous over $\Omega$, and so $u_{n,x_0} \nearrow u$ pointwise in $\Omega$, for some function $u$. Since $\mathcal{S}_t$ is bounded in $W^{2,p}(\Omega)$, it is precompact in $C^1(\overline{\Omega})$, and so by uniqueness of the pointwise limit, this convergence is actually $C^1$  and the function $u\in \mathcal{S}_t$. Assume for contradiction that for some $x_1\in \Omega$ we have $u(x_1)<\overline{u}(x_1)$. Then in the same way we can choose an increasing sequence $\tilde{u}_n\in\mathcal{S}_t$ such that $\tilde{u}_n(x_1) \to \overline{u}(x_1)$ and $\tilde{u}_n$ converges in $C^1$ to a solution $\tilde{u}$ such that $\tilde{u}(x_1)= \overline{u}(x_1)>u(x_1)$. Since for large $n$ we have $\tilde{u}_n(x_1)>u_n(x_1)$ and the solution set is ordered, we have $\tilde{u}_n\ge u_n$. Hence $\tilde{u}\ge u$, so $\tilde{u}(x_0)=u(x_0)$ since $u(x_0)=\overline{u}(x_0)$. We obtained two ordered solutions who touch, so by the strong maximum principle we get $\tilde u = u$, a contradiction.
\end{proof}

We now deduce a pivotal classification result. It says that the existence of three different solutions implies the existence of infinitely many (actually, a segment of) solutions.
\begin{lemma}
\label{Continuum_Solutions}
 If \eqref{Main_Decomp} has at least 3 solutions, then $\mathcal{S}_t$ is a line segment in $W^{2,p}(\Omega)$.
\end{lemma}
\begin{proof}
Suppose we have three distinct solutions of \eqref{Main_Decomp}. By Lemma \ref{ordered_sols}, they must be ordered, and so we write them as $\underline{u} < \Tilde{u} < \overline{u}$. By Corollary \ref{Ordered_Bdd}, we may assume $\underline{u}:= x \mapsto \inf_{u\in \mathcal{S}_t} u(x)$ and $\overline{u}:= x \mapsto \sup_{u\in \mathcal{S}_t} u(x)$.

  We claim that these three solutions are colinear (i.e. $\Tilde{u}-\underline{u} = k (\overline{u}-\underline{u})$ for some $k \in (0,1)$). Once this is proved,  the convexity of $F$ (via Corollary \ref{convex_subsuper})  immediately implies that the line segment connecting $\underline{u}$ and $\overline{u}$ is contained in $\mathcal{S}_t$, and is therefore equal to the entire set $\mathcal{S}_t$ by our assumption on $\underline{u}$ and $\overline{u}$.

   To prove the claim, write $u_{\alpha} = \overline{u} - \alpha(\overline{u} - \underline{u})$. By convexity, $F[u_{\alpha}] \leq g$ for any $\alpha \in [0,1]$. Notice that, although $u_\alpha$ is only a supersolution, it (along all other involved functions) nonetheless is an element of $W^{2,p}(\Omega)$ by construction.
 Let $\alpha^* = \sup\{ \alpha \in [0,1] : u_{\alpha} > \Tilde{u} \text{ in } \Omega \}$. Since $\underline{u} < \Tilde{u} < \overline{u}$ with strict inequality of the normal derivatives on the boundary (Remark~\ref{remhopf}), $\alpha^*$ is well-defined and $\alpha^*\in(0,1)$. By   continuity and $H1(F)$, we have
\begin{equation}
\mathcal{L}^+[\Tilde{u}-u_{\alpha^*}] \geq F[\Tilde{u}] - F[u_{\alpha^*}] \geq 0 \text{ in } \Omega, \qquad
        \Tilde{u}-u_{\alpha^*} \leq 0 \text{ in } \overline{\Omega}.
\end{equation}
So again either $\Tilde{u} = u_{\alpha^*}$ or $\Tilde{u} < u_{\alpha^*}$ with strict inequality of the normal derivatives on the boundary. The second alternative implies that $\Tilde{u} < u_{\alpha^*+\epsilon}$ for some $\epsilon>0$, which contradicts the definition of $\alpha^*$. Hence $\underline{u}$, $u_{\alpha^*}$, $\overline{u}$ are three different colinear solutions, and the proof is complete.
\end{proof}

Next, that {\it at least} two solutions of \eqref{Main_Decomp} exist when $t>t^*(h)$ follows from a classical degree argument which goes back to works of Amann (\cite{amann1972number},\cite{amann1976fixed}, see \cite{bandle2004solutions} or \cite{chang2005methods} for a general treatment of existence results based on degree theory) and was used in \cite{sirakovnon} too. Since the proof in \cite{sirakovnon} was not worded in a way that lends itself to easy comprehension by a reader of our text, and since we will need the construction later,  we give a (shortened but complete) proof here.

\begin{lemma}
\label{multiplicity}
    Assume there exist $\underline{v}, \overline{v} \in W^{2,p}({\Omega})$ satisfying, for $t > t^*$,
    \begin{equation}\label{cond1}
        \begin{cases}
         F[\overline{v}] \lneqq h + t\phi \lneqq F[\underline{v}], \qquad     \underline{v} < \overline{v} \text{ in } \Omega,\\
            \underline{v}=\overline{v}=0 \text{ on } \partial \Omega, \qquad
            \frac{\partial \underline{v}}{\partial \nu} < \frac{\partial \overline{v}}{\partial \nu} \text{ on } \partial\Omega.
        \end{cases}
    \end{equation}
 Then if we define
    \begin{equation}
        \mathcal{O} := \left\{ v \in C^1(\overline{\Omega}) : \underline{v} < v < \overline{v} \text{ in } \Omega, \ \frac{\partial \underline{v}}{\partial \nu} < \frac{\partial v}{\partial \nu} < \frac{\partial \overline{v}}{\partial \nu} \text{ on } \partial \Omega\right\},
    \end{equation}
    the sets $\mathcal{S}_t \cap \mathcal{O}$ and $\mathcal{S}_t \setminus \mathcal{O}$ are both nonempty, while $\mathcal{S}_t \cap \partial \mathcal{O} = \varnothing$.
\end{lemma}

The point of the construction is to find an open set in $C^1(\overline{\Omega})$ that separates the set of solutions of \eqref{Main_Decomp} into distinct components. That is to say, a set containing some (but not all) solutions of \eqref{Main_Decomp} in its interior, and without solutions on its boundary.

Before proving the lemma, we use it to verify that, when $t>t^*(h)$, the scenario described by Lemma \ref{Continuum_Solutions} leads to an immediate contradiction, so statement 2 in Theorem \ref{Main_Result} holds.

\begin{proof}[Proof of item 2 in Theorem \ref{Main_Result}] Fix $t>t^*$. Take $\overline{v}$ to be a solution of $(\mathcal{P}_{s})$ for some $t^* < s < t$ given by Theorem \ref{aprbdtstar} (so $\overline{v}$ is a strict supersolution of $(\mathcal{P}_{t})$)  and let $\underline{v}$ be the  strict subsolution given by Lemma \ref{subsol_exist} with $I = [s,t]$ (so $\underline{v}\le \overline{v}$). Then, Remark \ref{remhopf} applied to $\underline{v}$ and $\overline{v}$, implies this pair satisfies \eqref{cond1}, and by   Lemma \ref{multiplicity} there are at least two ordered solutions of \eqref{Main_Decomp}, one in $\mathcal{O}$ and one outside that set. Assume for contradiction that there exists a third solution. Then Lemma \ref{Continuum_Solutions} implies $\mathcal{S}_t$ is a line segment. Since this segment contains elements both of $\mathcal{O}$ and its complement, by continuity we obtain a solution on $\partial \mathcal{O}$, which contradicts the last statement of Lemma \ref{multiplicity}.
\end{proof}

\begin{proof}[Proof of Lemma \ref{multiplicity}]
 Notice that $\mathcal{O}$ is an open set in $C^1(\overline{\Omega})$. By the definition of $\mathcal{O}$ and the method of sub- and super-solutions (recalled above) we know that there exists a solution of \eqref{Main_Decomp} in $\overline{\mathcal{O}}$.
 Let us prove that there can be no solutions on $\partial \mathcal{O}$. Take any $u \in \mathcal{S}_t \cap \overline{\mathcal{O}}$. By $H1(F)$ we get
\begin{equation}
\begin{cases}
        \mathcal{L}^+[u-\overline{v}] \geq F[u] - F[\overline{v}] \geq 0 \text{ in } \Omega, \qquad
        u-\overline{v}\leq 0 \text{ in } \Omega\\
        \mathcal{L}^+[\underline{v}-u] \geq F[\underline{v}] - F[u] \geq 0 \text{ in } \Omega, \qquad
        \underline{v}-u \leq 0 \text{ in } \Omega,
        \end{cases}
\end{equation}
so that applying Remark \ref{remhopf} to both differences implies either $u \in \mathcal{O}$, or $u = \overline{v}$, or $u = \underline{v}$. The last two options are impossible due to our assumptions on $\underline{v}$ and $\overline{v}$ (specifically, that they are strict sub and supersolutions). Therefore, $u \in \mathcal{O}$; in particular $\mathcal{S}_t \cap \partial\mathcal{O}= \varnothing$.

 To find a solution outside $\mathcal{O}$, given any $v \in C^1(\overline{\Omega})$, we consider the auxiliary operator \begin{equation}F_v(M,p,u,x):= F(M,p,v(x),x)-\delta u,\end{equation}and the Dirichlet problem
\begin{equation}
\label{Aux_Dirichlet}
         F_v[u]:=F(D^2u, Du, v(x), x) - \delta u= g - \delta v \text{ in } \Omega, \qquad
        u = 0 \text{ on } \partial \Omega.
\end{equation}
As it is easy to check, H1($F_v$) and C($F_v$) hold for any fixed $v$, $F_v$ is proper, and $(F_v)_\infty(M,p,u,x) = F_\infty(M,p,0,x)- \delta u$. Hence there exists a unique solution $u$ of \eqref{Aux_Dirichlet} which belongs to $W^{2,p}(\Omega)$ (see for instance \cite{dongkrylov} and Theorem \ref{W2P}) and is such that
\begin{equation}\label{c1abd1}
\|u\|_{C^{1,\alpha}(\Omega)}\le C\left(n,p, \lambda,\Lambda, \gamma, \delta, \Omega, \|v\|_{L^\infty(\Omega)}\right)(\|g\|_{L^\infty(\Omega)} + \|v\|_{L^\infty(\Omega)}),
\end{equation}
where $\alpha= 1-n/p$.  Therefore, the operator $K_t$ that associates to $v$ the solution of \eqref{Aux_Dirichlet} for some fixed $h\in L^p(\Omega)$ and $g = h+t\phi$, i.e. \begin{equation}K_t(v) = F_v^{-1}[h+t\phi-\delta v],\end{equation} is well-defined and maps compactly $C^1_0(\overline{\Omega})$ to itself.
 By construction, solutions of \eqref{Main_Decomp} are exactly the fixed points of $K_t$.

 We claim that  $K_t(\overline{\mathcal{O}}) \subset \mathcal{O}$. Indeed, take $v \in \overline{\mathcal{O}}$ and set $u = K_t(v)$. Recalling the definition of $\mathcal{O}$ and \eqref{Aux_Dirichlet}, by H1($F$) and $v\le\overline{v}$ we get
\begin{align*}
    F(D^2u,Du,\overline{v},x) & =  F(D^2u,Du,\overline{v},x) +\delta \overline{v} -\delta \overline{v} \\
    &\geq F(D^2u,Du,v,x) + \delta v -\delta \overline{v} \\
    & = h+t\phi +\delta u -\delta \overline{v} \\
    & \gneqq F(D^2\overline{v},D\overline{v},\overline{v},x)+ \delta( u - \overline{v}),
\end{align*}
  that is, $F_{\overline{v}}[u]\gneqq F_{\overline{v}}[\overline{v}]$. Since $F_{\overline{v}}$ satisfies the comparison principle (see Corollary \ref{Comparison_Corollary}), $u \leq \overline{v}$. An analogous argument with $\underline{v}$ yields $u \geq \underline{v}$, and
   the Hopf-Oleinik strong maximum principle gives $u \in \mathcal{O}$, in the same way as before.

 Set $R_0= \max\{||\overline{v}||_{\infty},||\underline{v}||_{\infty}\}$ and let $\bar R$ be an upper bound for the right-hand side in \eqref{c1abd1} provided $\|v\|_\infty\le R_0$. In particular,  $K_t(\overline{\mathcal{O}})\subset \mathcal{B}_{\bar R}$. Set $R_1 = \max\{\bar R, \|\underline{v}\|_{C^1}, \|\overline{v}\|_{C^1}\}+1$. By this choice and the above claim  $K_t(\partial (\mathcal{O}\cap \mathcal{B}_{R_1}))\subset \mathcal{O}\cap \mathcal{B}_{R_1-1}$, hence $K_t$ has no fixed points on $\partial (\mathcal{O}\cap \mathcal{B}_{R_1})$.

Also, fix $t_0 = t^*(h)-1$ and take $\bar C $ to be an upper bound for the  $C^1$--norms of the solutions of $(\mathcal{P}_{s})$, for all $s\in [t_0,t]$ (by Theorem \ref{aprbd} and Theorem \ref{W2P} such a constant $\bar C$ exists). Define $R_2 = \max\{ \bar C+1, R_1+1\}$.

We claim that
\begin{equation}
    \text{deg}(I-K_t,\mathcal{O}\cap \mathcal{B}_{R_1},0) \neq \text{deg}(I-K_t,\mathcal{B}_{R_2},0).
\end{equation}

 Let us compute these degrees. First, fix $w \in \mathcal{O}\cap \mathcal{B}_{R_1-1}$ and consider the compact homotopy $A(s,v) = sK_t(v) + (1-s)w$. Note that, by construction, $\mathcal{O}$ is convex, and so, by  what we just proved, $A([0,1], \overline{\mathcal{O}\cap \mathcal{B}_{R_1}}) \subset \mathcal{O}\cap \mathcal{B}_{R_1}$, and there are no solutions of $v - A(s,v) = 0$ on $[0,1] \times \partial (\mathcal{O}\cap \mathcal{B}_{R_1})$. Therefore,
\begin{equation*}
    \text{deg}(I-K_t,\mathcal{O}\cap \mathcal{B}_{R_1},0) = \text{deg}(I-A(1,\cdot), \mathcal{O}\cap \mathcal{B}_{R_1},0) = \text{deg}(I-A(0,\cdot), \mathcal{O}\cap \mathcal{B}_{R_1},0) =  1,
\end{equation*}
since $A(0,\cdot)=w$ is a constant map and $w \in \mathcal{O}\cap \mathcal{B}_{R_1}$. This also gives an alternative proof of the fact we recalled above, that there is a solution in $\mathcal{O}\cap \mathcal{B}_{R_1}$.

Next, the map $B(s,v) = K_s(v)$ for $s \in [t_0,t]$ is clearly a compact homotopy connecting $K_{t_0}$ and $K_t$. Moreover, there are no solutions of $v - B(s,v) = 0$ on $\partial \mathcal{B}_{R_2}$, since such a fixed point would be an element of $\mathcal{S}_s$ with $C^1$ norm violating the bound $\bar C$. Then
\begin{equation*}
    \text{deg}(I-K_t,\mathcal{B}_{R_2},0) = \text{deg}(I-K_{t_0},\mathcal{B}_{R_2},0) = 0,
\end{equation*}
since $(\mathcal{P}_{t_0})$ has no solutions, and so $K_{t_0}$ has no fixed points. By the excision property of the degree,
\begin{equation*}
    \text{deg}(I-K_t,\mathcal{B}_{R_2}\setminus(\mathcal{O}\cap \mathcal{B}_{R_1}),0) = -1,
\end{equation*}
hence there exists a fixed point of $K_t$, and therefore a solution of $(\mathcal{P}_t)$, in $\mathcal{B}_{R_2}\setminus(\mathcal{O}\cap \mathcal{B}_{R_1})$. To conclude the proof of the lemma, it remains only to show that $\mathcal{S}_t \cap( \mathcal{B}_{R_2}\setminus(\mathcal{O}\cap \mathcal{B}_{R_1})) \subset \mathcal{S}_t \setminus \mathcal{O}$. If this were not the case, we would have a solution $u$ of $(\mathcal{P}_t)$ in $\mathcal{O} \setminus \mathcal{B}_{R_1}$, and therefore $u = K_t(u) \in K_t(\mathcal{O}) \subset \mathcal{B}_{R_1}$, which is a contradiction.
\end{proof}


Thus, we already know that $\mathcal{S}_{t}=\varnothing$ for $t<t^*=t^*(h)$, and that $\mathcal{S}_{t}=\{\underline{u}_t, \overline{u}_t\}$ with $\underline{u}_t< \overline{u}_t$ for $t>t^*$, where $t^*$ is a number that satisfies the bounds in Theorem \ref{aprbdtstar}.  Before moving on to the analysis of the case $t=t^*$, we  establish some properties of the two branches of $\mathcal{S}_{(t^*,\infty)}$, which will be used in that analysis and are also part of Theorem \ref{Main_Result}. We begin by the lower branch.

\begin{lemma}
\label{lower_curve_cont}
The map $t \mapsto \underline{u}_t$ is continuous from $(t^*,\infty)$ to $C^1(\overline{\Omega})$ and pointwise strictly decreasing, in the sense that $s<t$ implies $\underline{u}_s(x)> \underline{u}_t(x)$ for all  $x\in\Omega$.
\end{lemma}
\begin{proof}
    Fix $s\in(t^*,t)$. Consider $\underline{v}_t$ a subsolution given by Lemma \ref{subsol_exist} for $I = [t^*,t]$. Since $\underline{u}_s$ is a solution of $(\mathcal{P}_{s})$, it is a strict supersolution of $(\mathcal{P}_t)$ and $\underline{v}_t<\underline{u}_s$. Hence there exists a solution $\Tilde{u}$ in $\mathcal{S}_t$, satisfying $\underline{v}_t < \Tilde{u} < \underline{u}_s$. Since solutions of $(\mathcal{P}_t)$ are ordered by Lemma~\ref{ordered_sols}, and $\mathcal{S}_t = \{\underline{u}_t, \overline{u}_t\}$, either $\underline{u}_t < \overline{u}_t = \Tilde{u} < \underline{u}_s$, or $\underline{u}_t = \Tilde{u} < \underline{u}_s$. In either case, the second statement of the lemma is proved.

    To prove continuity, let $ s_n \rightarrow t_0>t^*$, and fix $s \in(t^*, t_0)$. Set $\overline{v} = \underline{u}_s$ (a strict supersolution to $\mathcal{P}_{t}$ for all $t > s$)  and take $\underline{v}$ a strict subsolution given by Lemma \ref{subsol_exist} for $ I = [s,t_0+1]$. As before, this pair of sub- and super-solutions satisfies the hypotheses of Lemma \ref{multiplicity} for all $t \in (s,t_0+1]$. Define $\mathcal{O}$ as in that lemma for this pair, then Lemma \ref{multiplicity}  implies that $\mathcal{O} \cap \mathcal{S}_{t} \neq \varnothing$, $\partial \mathcal{O} \cap \mathcal{S}_{t} = \varnothing$ and $\mathcal{S}_t \setminus \overline{\mathcal{O}} \neq \varnothing$, for all $t \in (s,t_0+1]$. By the monotonicity we just proved, $\underline{u}_t < \overline{v}$, and by the definition of $\underline{v}$ in Lemma \ref{subsol_exist}, $\underline{u}_t > \underline{v}$, both for all $t \in (s,t_0+1]$. Hence $\underline{u}_t \in \mathcal{O}$. Since there are only two elements of $\mathcal{S}_{t}$, this means $\overline{u}_{t} \notin \mathcal{O}$.
    \par For sufficiently large $n$, $s_n \in (s,t_0+1]$, and so $\underline{u}_{s_n} \in \mathcal{O}$. By Theorems \ref{aprbd} and \ref{W2P} $||\underline{u}_{s_n}||_{W^{2,p}(\Omega)}$ is bounded, and so Corollary \ref{Consistency_Cor} implies that a subsequence of $\underline{u}_{s_n} $ converges to a function $u$ which must be in $\mathcal{S}_{t_0} \cap \overline{\mathcal{O}} = \mathcal{S}_{t_0} \cap \mathcal{O}$ (since there can be no solutions on $\partial \mathcal{O}$), a set which we have just seen is the singleton $\{ \underline{u}_{t_0}\}$. This is true for every subsequence of $\{s_n\}_{n\in \mathbb{N}}$, and so $\underline{u}_{s_n} \rightarrow \underline{u}_{t_0}$, proving continuity. Recall that the convergence given by Corollary \ref{Consistency_Cor} is in $C^1(\overline{\Omega})$.
\end{proof}
\begin{remark}
\label{lower_curve_decreasing_remark}
    Note that the argument in the first paragraph in the above proof (that shows the map $t \mapsto \underline{u}_t$ is pointwise strictly decreasing) would not change if we worked with $s=t^* $ and an arbitrary $u \in \mathcal{S}_{t^*}$ instead of $\underline{u}_s$.
\end{remark}
\par The proof of Lemma \ref{lower_curve_cont}  can be adapted to show  that the set of maximal solutions is  a continuous curve too.
\begin{lemma}
\label{upper_curve_cont}
    The map $t \mapsto \overline{u}_t$ is continuous from $(t^*,\infty)$ to $C^1(\overline{\Omega})$. 
\end{lemma}
\begin{proof}
    Take $s_n \rightarrow t_0$ and $\mathcal{O}$ as in the proof of Lemma \ref{lower_curve_cont}. The argument there shows us that for large enough $n$, $\underline{u}_{s_n} \in \mathcal{O}$, and so, by Lemma \ref{multiplicity}, $\overline{u}_{s_n} \notin \mathcal{O}$. But by the same reasoning, $\overline{u}_{s_n} \rightarrow u \in \mathcal{S}_{t} \setminus \mathcal{O}$. Since there are only two elements of $\mathcal{S}_t$, and $\underline{u}_t \in \mathcal{O}$, this implies $\overline{u}_{s_n} \rightarrow \overline{u}_t$, again for every convergent subsequence, and so for the whole sequence  too.
\end{proof}

We do not know (and actually do not expect) that the upper curve $t \mapsto \overline{u}_t$ is also pointwise monotonous, as is the case for the lower curve, and strictly increasing.  Nonetheless, we are able to prove that the upper curve is strictly increasing on at least ``half of the domain".
\begin{lemma}
\label{upper_curve_nondecreasing}
    Let $t^*<s<t$. Then, the set $\Omega_{s,t}^+ = \{x\in\Omega\::\:\overline{u}_s(x) < \overline{u}_t(x)\}$ has measure greater than $|\Omega|/2$.
\end{lemma}
\begin{proof}
    First, suppose for contradiction that $\overline{u}_s \geq \overline{u}_t$ in $\Omega$. Take $\overline{v} = \overline{u}_s$ and $\underline{v}$ a subsolution given by Lemma \ref{subsol_exist} for $I = [s,t]$. This pair again satisfies the conditions of Lemma \ref{multiplicity}, which we apply. Since $\overline{u}_s \geq \overline{u}_t\ge \underline{v}$, we have $\overline{u}_t \in \overline{\mathcal{O}}$. We also know that $\underline{u}_t\in \mathcal{O}$ which, since $\mathcal{S}_t \setminus \overline{\mathcal{O}} \neq \varnothing$, implies the existence of a third solution to \eqref{Main_Decomp}, a contradiction with the already proven item 2 in Theorem \ref{Main_Result}.

    Then we know (by the continuity of $\overline{u}_s$ and $\overline{u}_t$), that $|\Omega_{s,t}^+|>0$. Suppose, by contradiction, that $|\Omega_{s,t}^+| \le|\Omega|/2$. $\text{D}_{F_{\infty}}(F)$ then implies
    \begin{equation}
            F_{\infty}[\overline{u}_t-\overline{u}_s] \geq (t-s)\phi \geq  0 \text{ in } \Omega_{s,t}^+, \qquad
            \overline{u}_t-\overline{u}_s = 0 \text{ on } \partial \Omega_{s,t}^+.
    \end{equation}
    Now, by our restrictions on $d>0$ and $\lambda_1^+(F) \geq -d$, we know that $\lambda_1^+(F_{\infty},\Omega_{s,t}^+) > 0$. Applying Theorem \ref{ABP} 4., we get
    \begin{equation*}
        \sup_{\Omega_{s,t}^+}( \overline{u}_t-\overline{u}_s) \leq C \sup_{\partial \Omega_{s,t}^+}( \overline{u}_t-\overline{u}_s) = 0
    \end{equation*}
    which is a contradiction with the definition of $\Omega_{s,t}^+$, and  concludes the proof of the lemma.
\end{proof}
\begin{remark}
\label{upper_curve_nondec_remark}
    Note that Lemma \ref{upper_curve_nondecreasing} still holds if we take $s=t^*$ and substitute $\overline{u}_s$ by an arbitrary $u \in \mathcal{S}_{t^*}$.
\end{remark}

Next, we describe the behaviour of the branches as $t \rightarrow \infty$. To do this, we rely on the fact that $F$ and $F_{\infty}$ have, by construction, the same qualitative behaviour for large $u$. Although we do not have monotonicity for the upper curve, it turns out that the weaker statement of Lemma \ref{upper_curve_nondecreasing} is enough to establish its asymptotic behaviour.
\begin{lemma}
    For any compact set $K \subset \Omega$, $\sup_K \underline{u}_t \rightarrow -\infty$ and $\inf_K \overline{u}_t \rightarrow +\infty$ as $t \rightarrow \infty$. Also, for $t$ large enough, $\underline{u}_t<0$ and $\overline{u}_t>0$ in $\Omega$.
\end{lemma}
\begin{proof}
    Denote by $u_t$ an element of $\mathcal{S}_t$ (either $\underline{u}_t$ or $\overline{u}_t$, since we are interested in large $t$). Set $w_t = u_t/t$. By $\text{D}_{F_{\infty}}(F)$ and U$(F)$, we have
    \begin{equation}
     F_{\infty}[w_t] -A_0/t =   \frac{F_{\infty}[u_t]-A_0}{t} \leq \frac{h+t\phi}{t} \leq \frac{F_{\infty}[u_t]}{t} = F_{\infty}[w_t].
    \end{equation}
    Hence $F_{\infty}[w_t] \rightarrow \phi$  in $L^p(\Omega)$. By Theorems \ref{aprbd} and \ref{W2P} the sequence $w_t$ is bounded in $W^{2,p}(\Omega)$ as $t\to\infty$.  Then, again by Theorem \ref{Consistency} and Theorem \ref{Compactness}, we get $w_t \rightarrow w$ up to a subsequence, in $C^1(\overline{\Omega})$, for $w$ satisfying
    \begin{equation}
    \label{Aux_Eigen?}
            F_{\infty}[w] = \phi \text{ in } \Omega, \qquad
            w = 0 \text{ on } \partial \Omega.
    \end{equation}

    Note that this is $(\mathcal{P}_{t})$ for $F=F_\infty$, $h=0$ and $t=1$. Therefore, by Lemma \ref{heq0} and Theorem \ref{Main_Result} 2.  there exist exactly two solutions of \eqref{Aux_Eigen?}. One of them is  $w^* := -\frac{1}{\lambda_1^+}\phi >0$, as we check immediately. By the second part of Theorem \ref{Positive_Existence} and the Hopf lemma there exists a strictly negative solution of \eqref{Aux_Eigen?} $w_*<0$ in $\Omega$ and $\frac{\partial w_*}{\partial\nu} < 0$ on $\partial \Omega$. Therefore, $w_t$ converges in $C^1$, up to subsequence, to $w_*<0$ or $w^*>0$ in $\Omega$.

     First, we deal with $\underline{u}_t$. For any fixed $t_0 > t^*$, we know $t>t_0$ implies $\underline{u}_t \leq \underline{u}_{t_0} \leq C$ for some constant depending on $t_0$ given by Theorem \ref{aprbd}. Dividing by $t$ and taking the limit, we get $\limsup_{t\rightarrow\infty}\underline{u}_t \leq 0$. This implies, by the above argument,
    \begin{equation*}
       \underline{w}_t =\frac{\underline{u}_t}{t} \rightarrow w_*
    \end{equation*}
    up to subsequence, in $C^1(\overline{\Omega})$. This is true for every subsequence, and therefore for the sequence itself. Since $w_* < 0$ in $\Omega$, we can, for any compact $K \subset \overline{\Omega}$, take $t$ large enough that
    \begin{equation*}
     \|\underline{w}_t-w_*\|_\infty\le  \frac{1}{2}\min_K |w_*|, \qquad\mbox{that is}, \qquad  \frac{\underline{u}_t}{t} \leq \frac{1}{2}w_* < 0
    \end{equation*}
    in $K$. Therefore, $\sup_K \underline{u}_t\leq \frac{1}{2}t\sup_K w_* \rightarrow -\infty$. Since the set of functions which are strictly negative in $\Omega$ and have a strictly negative normal derivative on the boundary is open in $C^1(\Omega)$ and $w_*$ belongs to that set, $\underline{w}_t$ belongs to it too for large $t$, so $\underline{u}_t<0$ in $\Omega$.

   Now, we move to the upper curve $t \mapsto \overline{u}_t$.  First, fix $t_0 > t^*$. From Lemma \ref{upper_curve_nondecreasing}, $t> t_0$ implies $|\Omega_{t_0,t}^+| \geq |\Omega|/2$. Take $\Tilde{K} \subset \Omega$ compact such that $|\Tilde{K}| \geq \frac{3}{4}|\Omega|$. This implies, for every $t>t_0$, $|\Tilde{K} \cap \Omega_{t_0,t}^+ | \geq |\Omega|/4 >0$.

     By the argument laid out above, we know that for any sequence $t_n\rightarrow \infty$, the sequence $\overline{u}_{t_n}$ subconverges to a solution of \eqref{Aux_Eigen?}, i.e. $w_*$ or $w^*$. Since the measure of $|\Tilde{K} \cap \Omega_{t_0,t}^+ |$ is always positive, we can always pick $x_n \in \Tilde{K}$ such that
    \begin{equation*}
        \frac{\overline{u}_{t_n}(x_n) - \overline{u}_{t_0}(x_n)}{t} > 0.
    \end{equation*}
    \par Then, passing to a further subsequence if necessary, $x_n \rightarrow x_{\infty} \in \Tilde{K} \subset \Omega$, and so
    \begin{equation}
        w(x_{\infty}) \geq \liminf_{n\rightarrow \infty} \frac{\overline{u}_{t_n}(x_n) - \overline{u}_{t_0}(x_n)}{t_n} \geq 0.
    \end{equation}
    \par Since $w_* < 0 $ in $\Omega$, we conclude that $w = w^*$. This is true for any subsequence, and so $\overline{u}_t/t \rightarrow w^*$ in $C^0(\overline{\Omega})$. By the same argument as in the lower curve's case, then, we get, for any compact $K \subset \Omega$ and large enough $t$, $\inf_K \overline{u}_t \geq \frac{1}{2}t\inf_K w^* \rightarrow \infty$ and $\overline{u}_t >0$ in $\Omega$.
\end{proof}

 We now move to the analysis of the case $t = t^*$.

 First, we note that by applying Theorems \ref{aprbd}-\ref{W2P}, Corollary \ref{Consistency_Cor} and Lemma~\ref{lower_curve_cont}, we get $\underline{u}_t \rightarrow \underline{u} := \sup_{t > t^*} \underline{u}_t$ in $C^1(\overline{\Omega})$ as $t \searrow t^*$. Here, the convergence obtained from Corollary \ref{Consistency_Cor} is true for every subsequence due to the monotonicity of $\underline{u}_t$ given by Lemma~\ref{lower_curve_cont}. We claim that $\underline{u}(x)= \inf_{u \in \mathcal{S}_{t^*}} u(x) =:u_*(x)$.

 To see this, recall that by Corollary \ref{Ordered_Bdd} the solution set $\mathcal{S}_{t^*}$ is strictly ordered and bounded, and $u_* \in \mathcal{S}_{t^*}$. We argue by contradiction: suppose $\underline{u} > u_*$. The $C^1(\overline{\Omega})$--convergence implies there exists some $t>t^*$ such that $\underline{u}_t > u_*$ at some point in $\Omega$. This contradicts the second statement of Lemma \ref{lower_curve_cont} (see also Remark \ref{lower_curve_decreasing_remark} after that lemma).

Now, we move to the upper curve. By Corollary \ref{Ordered_Bdd}, the set of solutions $\mathcal{S}_{t^*}$ is ordered and bounded, and if $u^* := x \mapsto  \sup_{u \in \mathcal{S}_{t^*}}u(x)$, $u^* \in \mathcal{S}_{t^*}$. We wish to verify that $u^*$ is the limit of the upper branch of solutions as $t \rightarrow t^*$. Again by Theorems \ref{aprbd}-\ref{W2P} and Corollary \ref{Consistency_Cor} any sequence $t_n\rightarrow t^*$ has a subsequence such that $\overline{u}_{t_n} \rightarrow u \in \mathcal{S}_{t^*}$ in $C^1(\overline{\Omega})$. By the definition of $u^*$, $u \leq u^*$. It suffices to discard the strict inequality. Suppose, by contradiction, that   $u < u^*$ at some point in $\Omega$. Then, Theorem \ref{SMP} implies $u<u^*$ in $\Omega$ and $\frac{\partial {u}}{\partial \nu} < \frac{\partial {u^*}}{\partial \nu}$ on $\partial\Omega$. By $C^1$--convergence, then, there is $n$ large enough so that $\overline{u}_{t_{n}} < u^*$ in $\Omega$. This contradicts  Lemma \ref{upper_curve_nondecreasing} (see Remark \ref{upper_curve_nondec_remark}), so $u = u^*$. Since this is true for every subsequence of $t_n$, we have proved that  $ \lim_{t\rightarrow t^*} \overline{u}_t = u^*$.

Now, we know that there are two ordered continuous branches of solutions of $(\mathcal{P}_t)$, $\underline{u}_t$ and $\overline{u}_t$ for $t>t^*$, stemming from some ``root'' solutions of $(\mathcal{P}_{t^*})$ $u_*$ and $u^*$, respectively the infimum and the supremum of $\mathcal{S}_{t^*}$. If $u_* = u^*$, we are done; however, if they are not, Lemma \ref{ordered_sols} implies they are ordered (and therefore $u_* < u^*$), but nothing more.  We need to somehow argue that the set $\mathcal{S}_{\R}$ is a connected curve. This will be achieved with the help of the next lemma, which provides  a continuity argument. 
\begin{lemma}
\label{notisolated}
    If $u_* \neq u^*$, $u_*$ is not an isolated solution of $(\mathcal{P}_{t^*})$ in the $C^1(\overline{\Omega})$--topology.
\end{lemma}
\begin{proof}
    Assume for contradiction that there exists some $C^1(\overline{\Omega})$--ball $\mathcal{B}_{2r}(u_*)$ such that \begin{equation}\label{assum1}
    \mathcal{S}_{t^*} \cap \mathcal{B}_{2r}(u_*) = \{u_*\}.
    \end{equation}
    Recall that $u^*> u_*$ in $\Omega$ and $\frac{\partial u_*}{\partial \nu} < \frac{\partial u^*}{\partial \nu}$ on $\partial \Omega$.
      Set $u_{\alpha} = u_* + \alpha (u^* - u_*)$. By Corollary \ref{convex_subsuper}, $u_{\alpha}$ is a supersolution of $(\mathcal{P}_{t^*})$ for $\alpha \in [0,1]$ and a subsolution for $\alpha \in \R\setminus [0,1]$. Take $\varepsilon>0$ small enough  that $\overline{v} = u_{\varepsilon}$ and $\underline{v} = u_{-\varepsilon}$ be elements of $\mathcal{B}_r(u_*)$. By our assumption \eqref{assum1}, this implies they are strict sub and supersolutions, respectively. 

     Next, define $\mathcal{O}$ and $K_t = K_{t^*}$ exactly as in the proof of Lemma \ref{multiplicity}, for these $\underline{v}$, $\overline{v}$. Also, define $\bar R$ as in that lemma, and $\Tilde{R_1} = \max\{\bar R, \|\underline{v}\|_{C^1}, \|\overline{v}\|_{C^1},||u_*||_{C^1}+r\}+1$, so that, in particular, $\mathcal{B}_r(u_*) \subset \mathcal{B}_{\Tilde{R_1}}$. For the same reasons as given in the proof of Lemma \ref{multiplicity}, we know that $K_{t^*}(\overline{\mathcal{O}}) \subset \mathcal{O}$, from which we can deduce that $K_{t^*}(\partial (\mathcal{O} \cap \mathcal{B}_{\Tilde{R_1}})) \subset \mathcal{O} \cap \mathcal{B}_{\Tilde{R_1}}$. Therefore, again defining $A(s,v) = sK_{t^*}(v) +(1-s)w$, for some fixed $w \in \mathcal{O}\cap \mathcal{B}_{R_1}$,
    \begin{equation*}
    \text{deg}(I-K_{t^*},\mathcal{O} \cap \mathcal{B}_{\Tilde{R_1}},0) = \text{deg}(I-A(1,\cdot), \mathcal{O} \cap \mathcal{B}_{\Tilde{R_1}},0) =\text{deg}(I-A(0,\cdot), \mathcal{O} \cap \mathcal{B}_{\Tilde{R_1}},0)  = 1.
\end{equation*}

However, if we now define the compact homotopy $B(s,v) = K_s(v)$ for $s \in [t^*-1,t^*]$, we know that there can be no solutions of $u - B(s,u) = 0$ on $[t^*-1,t^*] \times \partial (\mathcal{O} \cap \mathcal{B}_{r}(u_*))$, since $\mathcal{S}_t$ is empty for any $t< t^*$, $u_*\in \mathcal{O}$, \eqref{assum1} holds, and $\partial (\mathcal{O} \cap \mathcal{B}_r(u_*))\subset \mathcal{B}_{2r}(u_*)$. Then, we can calculate the degree of $I - K_{t^*}$ on $\mathcal{O} \cap \mathcal{B}_r(u_*)$, and find
\begin{equation*}
    \text{deg}(I-K_{t^*},\mathcal{O}\cap \mathcal{B}_r(u_*),0) = \text{deg}(I-K_{t^*-1}, \mathcal{O}\cap \mathcal{B}_r(u_*),0) = 0,
\end{equation*}
since \eqref{Main_Decomp} has no solutions for $t = t^*-1$. By the excision property of the degree, then, we can now calculate the degree of $I - K_{t^*}$ on $(\mathcal{O}\cap \mathcal{B}_{\Tilde{R_1}}) \setminus (\mathcal{O}\cap \mathcal{B}_r(u_*)) = \mathcal{O} \cap (\mathcal{B}_{\Tilde{R_1}} \setminus \mathcal{B}_r(u_*))$ and obtain
\begin{equation*}
    \text{deg}(I-K_{t^*},\mathcal{O} \cap (\mathcal{B}_{\Tilde{R_1}} \setminus \mathcal{B}_r(u_*)),0) = 1
\end{equation*}
which implies the existence of a solution of $\mathcal{P}_t$ on that set. Since $u_* \in \mathcal{B}_r(u_*)$ and $u^* \notin \mathcal{O}$, this is a third, distinct solution of $\mathcal{P}_t$, so Lemma \ref{Continuum_Solutions} yields the desired contradiction.
\end{proof}

\begin{proof}[Proof of item 3 in Theorem \ref{Main_Result}]
From the claims we have already proven regarding the upper and lower solution branches, we know
\begin{equation}\label{abov}
    \begin{cases}
        \lim_{t\searrow t^*}\underline{u}_t = \inf_{u\in \mathcal{S}_{t^*}}u = u_*\in \mathcal{S}_{t^*} \\
        \lim_{t\searrow t^*}\overline{u}_t = \sup_{u\in \mathcal{S}_{t^*}}u = u^* \in \mathcal{S}_{t^*}
    \end{cases}
\end{equation}

If $u_*=u^*$, we are done. If not, by Lemma \ref{ordered_sols}, $u_*<u^*$ in $\Omega$. From Lemma \ref{notisolated}, $u_*$ is not an isolated element of $\mathcal{S}_{t^*}$, and so, in particular, $(\mathcal{P}_{t^*})$ has at least three distinct solutions. Lemma \ref{Continuum_Solutions} and \eqref{abov}  imply $\mathcal{S}_{t^*} = \{ u_*+\alpha (u^*-u_*) : \alpha \in [0,1]\}$ as desired.

If sC($F$) holds and $u_*<u^*$, $\mathcal{S}_{t^*} = \{ u_*+\alpha (u^*-u_*) : \alpha \in [0,1]\}$ immediately implies a contradiction. Indeed, for any $\alpha \in (0,1)$, $u_{\alpha} = u_*+\alpha (u^*-u_*) \in C^1(\overline{\Omega})$, and therefore, for some $x_0$ sufficiently close to $\partial\Omega$, $u_{\alpha}(x_0) \in (0, \eta)$ or $u_{\alpha}(x_0) \in (-\eta,0)$ for all $\alpha \in [0,\epsilon]$ and some $\epsilon>0$, which yields by sC($F$)
\begin{equation}
    h(x_0) +t^*\phi(x_0) = F[u_{\alpha}](x_0) < (1-\alpha)F[u_*](x_0) + \alpha F[u^*](x_0) = h(x_0) + t^*\phi(x_0).
\end{equation}
Therefore, sC($F$) implies $u_*=u^*$.
\end{proof}

Finally, we show that the map $h \mapsto t^*(h)$ is continuous. Assume $h_n \rightarrow h$ in $L^p(\Omega)$ and denote $t^*_n = t^*(h_n)$, $t^* = t^*(h)$. We want to prove that $t^*_n \rightarrow t^*$.

First, the sequence $\{t_n^*\}$ is bounded, by Theorem \ref{aprbdtstar}. Then, $t_n^* \rightarrow a \in \R$ up to some subsequence, which we denote again by $t_n^*$. Theorems \ref{aprbd}-\ref{W2P} and
Corollary \ref{Consistency_Cor} imply (taking a further subsequence if necessary) $u_n \rightarrow u$ in $C^1(\overline{\Omega})$, and $u$ solves the limit problem $(\mathcal{P}_{a,h})$. Hence $a \geq t^*(h)$. Suppose by contradiction $a > t^*$. Fix $\varepsilon>0$ such that $a - 3 \varepsilon > t^*$. Therefore there exists $\overline{v}$ which solves $(\mathcal{P}_{a-3\varepsilon,h})$. Also, fix $\underline{v}$ a strict subsolution of $(\mathcal{P}_{t})$ for $t\in [t^*, a]$ given by Lemma \ref{subsol_exist}. In particular, $\underline{v}<\overline{v}$ in $\Omega$ and $\frac{\partial \underline{v}}{\partial \nu} < \frac{\partial \overline{v}}{\partial \nu}$ on $\partial \Omega$.

 Let $w_n$ be a solution of
\begin{equation}
        F_{\infty}[w_n] - \delta w_n = h_n - h \text{ in } \Omega, \qquad
        w_n = 0 \text{ on } \partial \Omega.
\end{equation}
By Remark \ref{properex}, Theorem \ref{W2P} and Sobolev compact inclusions, $w_n \rightarrow 0$ in $C^1(\overline{\Omega})$. In particular, if $n$ is fixed large enough, we have $\overline{v}+w_n>\underline{v}$ in $\Omega$.

Further, applying $\text{D}_{F_{\infty}}(F)$,
\begin{equation}
    F[\overline{v} + w_n]\leq F[\overline{v}] +F_{\infty}[w_n] = h + (a-3\varepsilon)\phi + \delta w_n + h_n - h = h_n + (a-3\varepsilon)\phi + \delta w_n.
\end{equation}
Since $w_n \rightarrow 0$ in $C^1(\overline{\Omega})$, taking $n$ large enough we have $\delta w_n \leq c_0d\le \varepsilon\phi$, so that
\begin{equation}
    F[\overline{v}+w_n] \leq h_n + (a-2\varepsilon)\phi.
\end{equation}
Taking $n$ even larger if necessary, we have $t_n^* > a-\varepsilon$, hence
\begin{equation}
    F[\overline{v}+w_n] \leq h_n + (t_n^* -\varepsilon)\phi,
\end{equation}
i.e. $\overline{v} + w_n$ is a supersolution to $(\mathcal{P}_{t_n^*-\varepsilon,h_n})$. Hence we can find a solution to this problem,  which contradicts the definition of $t_n^* = t^*(h_n)$. Then, every convergent subsequence of $\{t_n^*\}$ converges to $t^*$. Since the sequence is bounded, this implies $t_n^* \rightarrow t^*$, and so the map $h \mapsto t^*(h)$ is continuous from $L^p(\Omega)$ to $\R$.

\section{Appendix}

{\small
In this appendix we record a few elementary properties of convex functions.

\begin{prop}\label{app1}
Let $f:\mathbb{R}^m\to\mathbb{R}$ be a convex and Lipschitz continuous function, with Lipschitz constant $L$. Assume also $f(0)=0$. Define, for $y\in\mathbb{R}^m$,
$$
f_\infty(y) = \sup_{t>0}\frac{f(ty)}{t}\;\left(\le L|y|<\infty\right).
$$
Then
\begin{enumerate}
\item $f_\infty(y) = \displaystyle \lim_{t\to\infty}\frac{f(ty)}{t}$;
\item  $f_\infty :\mathbb{R}^m\to\mathbb{R}$ is convex and Lipschitz, with Lipschitz constant $L$;
\item $f_\infty(ty) = tf_\infty(y)$, for $t>0$, $y\in\mathbb{R}^m$;
\item $f(y+z) - f(y)\le f_\infty(z)$, for  $y,z\in\mathbb{R}^m$.
\end{enumerate}
\end{prop}

\begin{proof}    Fix $y\in\mathbb{R}^m$ and define, for $t>0$, $g(t) =t^{-1} f(ty)$. Recall that $f(0) = 0$, and so, by convexity, $f(ay) \geq af(y)$ for $a\geq 1$. Then, for $0< s \leq t$,
    \begin{equation}
        g(t) = \frac{1}{t} f(ty) = \frac{1}{t}f\left(s \frac{t}{s}y\right) \geq \frac{1}{t}\frac{t}{s}f(sy) = g(s).
    \end{equation}
    so  $g$ is monotone and bounded, and hence \[f_{\infty}(y) = \sup_{t\in \R^+} g(t) = \lim_{t \rightarrow \infty} g(t).\]
  Properties 2. and 3. follow from 1. To prove 4. note that, by convexity, for any $s \in (0,1)$,
    \begin{equation}
        \begin{split}
            f(v+w)-f(v) \leq sf(v)+(1-s)f\left(\frac{1}{1-s}w + v\right) -f(v)
        \end{split}
    \end{equation}
    and so, passing to the limit,
    \begin{equation}
        \begin{split}
            f(v+w)-f(v) &\leq -f(v) + \lim_{s\rightarrow 1} sf(v)+\lim_{s\rightarrow 1}(1-s)f\left(\frac{1}{1-s}w + v\right)  \\
            &= 0 + \lim_{s\rightarrow 1} (1-s)f\left(\frac{1}{1-s}w + v\right) \\
            & \leq \lim_{s\rightarrow 1} (1-s)f_{\infty}\left(\frac{1}{1-s}w + v\right) \\
            &= \lim_{s\rightarrow 1} f_{\infty}\left(w + (1-s)v\right)=f_{\infty}(w).
        \end{split}
    \end{equation}
    Here, we used the facts that $f(u) \leq f_{\infty}(u)$ and that $f_{\infty}$ is positively $1$-homogeneous. \end{proof}

\begin{prop}\label{app2}
Let $f:\mathbb{R}^m\to\mathbb{R}$ be  convex, and $f(v)= f(w) =a$, for some $v,w\in \mathbb{R}^m$. Set $g(t) = f(v+t(w-v))$ for $t\in \mathbb{R}$. Then
\begin{enumerate}
\item $g(t)\le a$ for $t\in [0,1]$ and $g(t)\ge a$ for $t\in (-\infty,0]\cap [1,\infty)$
\item if there exists $t_0\in (0,1)$ such that $g(t_0)=a$ then $g(t)=a$ for all $t\in [0,1]$.
\end{enumerate}
\end{prop}

   \begin{proof}
       1.  The first statement is a straightforward application of the definition of convexity. For the second, if $t \leq 0$, write
        \begin{align*}
            a = f(x) = f \left( \frac{1}{1-t}(x+t(y-x)) + \frac{-t}{1-t}y  \right) &\leq \frac{1}{1-t}f(x+t(y-x)) + \frac{-t}{1-t}f(y) \\
            &= \frac{1}{1-t}f(x+t(y-x)) + \frac{-t}{1-t}a
        \end{align*}
        which, after rearrangement, yields $f(x+t(y-x)) \geq a$. If $t \geq 1$, similarly
        \begin{equation*}
            a = f(y) = f \left( \frac{1}{t}(x+t(y-x)) + \frac{t-1}{t}x  \right) \leq \frac{1}{t}f(x+t(y-x)) + \frac{t-1}{t}a
        \end{equation*}
        which again yields the same conclusion.

        2. Fix $t_1\in(0,1)$. Say $t_1<t_0$, then there is $s\in(0,1)$ such that $t_0=st_1+(1-s)$. Hence
        $$
        a= g(t_0) \le sg(t_1) + (1-s)g(1) = s g(t_1) + (1-s)a,$$
        so $g(t_1)\ge a$, and by 1. $g(t_1)= a$.
    \end{proof}
}

\bibliographystyle{abbrv}
\bibliography{biblio1}

 Maria Luísa Pasinato, Institut de Mathématiques,
         Université Paul Sabatier,
         118 route de Narbonne,
         31062 Toulouse Cedex, France
         \smallskip
         
         Boyan Sirakov, Departamento de Matem\'atica, PUC-Rio, Rua Marquês de São Vicente 225, CEP 22451-090, Rio de Janeiro, Brazil
\end{document}